\documentclass{amsproc}

\usepackage{fullpage}
\usepackage[utf8]{inputenc}
\usepackage[T1]{fontenc}
\usepackage{graphicx,subfigure}
\usepackage{amsmath,amsfonts,amssymb,mathtools}
\usepackage{stmaryrd}
\usepackage{mathrsfs,dsfont}
\usepackage{amsthm}
\usepackage{mathabx}
\usepackage{tabularx} 
\usepackage{float}
\usepackage{amscd}

\usepackage{hyperref}

\usepackage{enumerate}

\usepackage{color}

\newcommand{\E}{\mathbb{E}}
\newcommand{\R}{\mathbb{R}}
\newcommand{\N}{\mathbb{N}}

\newtheorem{theo}{Theorem}[section]

\newtheorem{propo}[theo]{Proposition}
\newtheorem{lemma}[theo]{Lemma}

\newtheorem{ass}{Assumption}

\usepackage{algorithm}
\usepackage{algorithmic}

\begin{document}

\title{Uniform strong and weak error estimates for numerical schemes applied to multiscale SDEs in a Smoluchowski--Kramers diffusion approximation regime}

\author{Charles-Edouard Br\'ehier}
\address{Univ Lyon, Université Claude Bernard Lyon 1, CNRS UMR 5208, Institut Camille Jordan, 43 blvd. du 11 novembre 1918, F--69622 Villeurbanne cedex, France}
\email{brehier@math.univ-lyon1.fr}

\begin{abstract}
We study a family of numerical schemes applied to a class of multiscale systems of stochastic differential equations. When the time scale separation parameter vanishes, a well-known Smoluchowski--Kramers diffusion approximation result states that the slow component of the considered system converges to the solution of a standard It\^o stochastic differential equation. We propose and analyse schemes for strong and weak effective approximation of the slow component. Such schemes satisfy an asymptotic preserving property and generalize the methods proposed in the recent article~\cite{BR}. We fill a gap in the analysis of these schemes and prove strong and weak error estimates, which are uniform with respect to the time scale separation parameter.
\end{abstract}

\maketitle

\section{Introduction}

In this article, we consider multiscale systems of stochastic differential equations of the type
\begin{equation}\label{eq:SDEintro}
\left\lbrace
\begin{aligned}
dq^\epsilon(t)&=\frac{p^\epsilon(t)}{\epsilon}dt\\
dp^\epsilon(t)&=-\frac{p^\epsilon(t)}{\epsilon^2}dt+\frac{f(q^\epsilon(t))}{\epsilon}dt+\frac{\sigma(q^\epsilon(t))}{\epsilon}d\beta(t),
\end{aligned}
\right.
\end{equation}
where $\epsilon\in(0,\epsilon_0)$ is a time-scale separation parameter. The unknowns $q^\epsilon(t)$ and $p^\epsilon(t)$ and the Wiener process $\beta$ take values in $\R^d$, and the mapping $f$ and $\sigma$ satisfy appropriate regularity conditions. The objective of this article is to study numerical schemes for the approximation of the component $Q^\epsilon$, for arbitrary values of the time-scale separation parameter $\epsilon$, in particular when it vanishes. This is not a trivial task since the component $p^\epsilon$ evolves at the fast time scale $t/\epsilon^2$, and a crude discretization would impose stringent conditions on the time-step size $\Delta t$.

It is a well-known result in the analysis of multiscale stochastic systems that $q^\epsilon(t)$ converges, at least in distribution, when $\epsilon\to 0$, to $q^0(t)$, for all $t\ge 0$, where $X^0$ is the solution of the stochastic differential equation
\begin{equation}\label{eq:limitingSDEintro}
dq^0(t)=f(q^0(t))dt+\sigma(q^0(t))d\beta(t)
\end{equation}
where the noise is interpreted in the sense of It\^o. We refer for instance to~\cite[Chapter~11]{PavliotisStuart} for a description of this thype of convergence result, and see Proposition~\ref{propo:cv} below for a precise statement, where convergence is understood in a stronger sense than convergence in distribution. The convergence result $q^\epsilon\to q^0$ is often called a Smoluchowski--Kramers diffusion approximation result in the literature. If $\sigma$ is constant and equal to the identity, and if $f=-\nabla V$ for some potential energy function $V:\R^d\to\R$, the SDE system~\eqref{eq:SDEintro} describes the Langevin dynamics, whereas the SDE~\eqref{eq:limitingSDEintro} describes the overdamped Langevin dynamics, see for instance~\cite[Sections~2.2.3 and~2.2.4]{LelievreRoussetStoltz}, and also the recent article~\cite{RoussetXuZitt} and references therein.

In order to define numerical schemes which perform better than crude methods when $\epsilon$ varies and may vanish, it is relevant to resort to the notion of asymptotic preserving schemes as studied in the recent article~\cite{BR}: if $\Delta t=T/N$ denotes the time-step size with given $T\in(0,\infty)$ and $N\in\N$, one has a commutative diagram property
\[
\begin{CD}
q_N^{\epsilon,\Delta t}     @>{N \to \infty}>> q^\epsilon(T) \\
@VV{\epsilon\to 0}V        @VV{\epsilon\to 0}V\\
q_N^{0,\Delta t}     @>{N \to \infty}>> q^0(T),
\end{CD}
\]
where $\bigl(q_n^{\epsilon,\Delta t},p_n^{\epsilon,\Delta t}\bigr)_{0\le n\le N}$ is the scheme for given values of $\epsilon$ and $\Delta t$, and one needs to check that
\begin{itemize}
\item the scheme is consistent for any value of $\epsilon>0$ when $\Delta t\to 0$,
\item there exists a limiting scheme $\bigl(q_n^{0,\Delta t}\bigr)_{0\le n\le N}$ when $\epsilon\to 0$ for any value of $\Delta t>0$,
\item the limiting scheme is consistent with the limiting equation when $\Delta t\to 0$.
\end{itemize}
As explained in~\cite{BR}, the last property may fail to hold for some crude methods. However in the situation considered in this article there are no such subtleties for instance related to the interpretation of the noise.

In this article, we study numerical schemes and obtain strong and weak error estimates which are uniform with respect to the time-scale separation parameter $\epsilon\in(0,\epsilon_0)$, meaning that $q^\epsilon(T)$ can be approximated by $q_N^{\epsilon,\Delta t}$ with a cost which is independent of $\epsilon$. On the one hand, for the numerical scheme
\begin{equation}\label{eq:scheme-strong-intro}
\left\lbrace
\begin{aligned}
q_{n+1}^{\epsilon,\Delta t}&=q_n^{\epsilon,\Delta t}+\frac{\Delta t}{\epsilon}p_{n+1}^{\epsilon,\Delta t}\\
p_{n+1}^{\epsilon,\Delta t}&=p_n^{\epsilon,\Delta t}-\frac{\Delta t}{\epsilon^2}p_{n+1}^{\epsilon,\Delta t}+\frac{\Delta tf(q_n^{\epsilon,\Delta t})}{\epsilon}+\frac{\sigma(q_n^{\epsilon,\Delta t})}{\epsilon}\Delta\beta_n,
\end{aligned}
\right.
\end{equation}
one obtains uniform strong error estimates
\begin{equation}\label{eq:strong-intro}
\underset{\epsilon\in(0,\epsilon_0)}\sup~\bigl(\E[|q_N^{\epsilon,\Delta t}-q^\epsilon(N\Delta t)|^2]\bigr)^{\frac12}\le C(T)\Delta t,
\end{equation}
see Theorem~\ref{theo:strong} for a precise statement. On the other hand, for the numerical scheme
\begin{equation}\label{eq:scheme-weak-intro}
\left\lbrace
\begin{aligned}
q_{n+1}^{\epsilon,\Delta t}&=q_n^{\epsilon,\Delta t}+\epsilon\bigl(1-e^{-\frac{\Delta t}{\epsilon^2}}\bigr)p_n^{\epsilon,\Delta t}+\bigl(\Delta t-\epsilon^2(1-e^{-\frac{\Delta t}{\epsilon^2}})\bigr)f(q_n^{\epsilon,\Delta t})\\
&+\sigma(q_n^{\epsilon,\Delta t})(\beta(t_{n+1})-\beta(t_n))-\sigma(q_n^{\epsilon,\Delta t})\int_{t_n}^{t_{n+1}}e^{-\frac{t_{n+1}-s}{\epsilon^2}}d\beta(s)\\
p_{n+1}^{\epsilon,\Delta t}&=e^{-\frac{\Delta t}{\epsilon^2}}p_n^{\epsilon,\Delta t}+\epsilon(1-e^{-\frac{\Delta t}{\epsilon^2}})f(q_n^{\epsilon,\Delta t})+\frac{1}{\epsilon}\sigma(q_n^{\epsilon,\Delta t})\int_{t_n}^{t_{n+1}}e^{-\frac{t_{n+1}-s}{\epsilon^2}}d\beta(s),
\end{aligned}
\right.
\end{equation}
and functions $\varphi:\R^d\to\R$ of class $\mathcal{C}^3$, one obtains weak error estimates
\begin{align}
\big|\E[\varphi(q_N^{\epsilon,\Delta t})]-\E[\varphi(q^\epsilon(T))]\big|&\le C(T,\varphi)\bigl(\Delta t+\epsilon\bigr),\label{eq:weak-intro1}\\
\underset{\epsilon\in(0,\epsilon_0)}\sup~\big|\E[\varphi(q_N^{\epsilon,\Delta t})]-\E[\varphi(q^\epsilon(T))]\big|&\le C(T,\varphi)\Delta t^{\frac12}.\label{eq:weak-intro2}
\end{align}
See Theorem~\ref{theo:weak} for a precise statement, and Proposition~\ref{propo:SDEtilde} for explanations concerning the construction of the numerical scheme~\eqref{eq:scheme-weak-intro}.

The uniform weak error estimate~\eqref{eq:weak-intro2} may be suboptimal: indeed one only obtains order $1/2$, and therefore~\eqref{eq:weak-intro2} is a straightforward corollary of the uniform strong error estimate~\eqref{eq:strong-intro}. The weak error estimate~\eqref{eq:weak-intro1} is not uniform with respect to $\epsilon$ but may be more precise in regimes where $\epsilon$ is negligible compared with $\Delta t$. The weak error estimate~\eqref{eq:weak-intro1} is also similar to error bounds which may be obtained for different multiscale numerical schemes, for instance based on the Heterogeneous Multiscale Method. Improving the uniform weak error estimate~\eqref{eq:weak-intro2} to obtain order $1$ is left for future works.

To the best of our knowledge, the strong error estimates~\eqref{eq:strong-intro} and the weak error estimates~\eqref{eq:weak-intro1}--\eqref{eq:weak-intro2} have not been obtained previously and our results thus fill a gap in the literature. These proofs require delicate and non trivial arguments. On the one hand, proving~\eqref{eq:strong-intro} is based on an appropriate change of unknowns and analysis of multiple error terms. On the other hand, proving~\eqref{eq:weak-intro1}--\eqref{eq:weak-intro2} is based on the standard approach using solutions of Kolmogorov equations for weak error analysis. Proposition~\ref{propo:Kolmogorov} gives the required bounds on the derivatives of these solutions, with a careful analysis of the dependence with respect to $\epsilon$.

Note that the recent preprint~\cite{preprintWZ} is also concerned with the proof of uniform (strong) error estimates for a class of multiscale SDE systems in a diffusion approximation regime. However, the structure of the systems, the results and the techniques of proof are substantially different, which justifies to perform the analysis in separate articles.

The analysis of numerical methods for multiscale stochastic differential equations is an active research area. The recent articles~\cite{ACLM} and~\cite{Laurent} propose uniformly accurate methods for SDE systems which are different from~\eqref{eq:SDEintro} considered in this article. The recent article~\cite{BR} has introduced a notion of asymptotic preserving schemes which applies to~\eqref{eq:SDEintro}, and some uniform error estimates were proved for SDE systems in an averaging regime. We also refer to the PhD thesis~\cite{RR-thesis} for supplementary results and numerical experiments. In this article, as already mentioned, we fill a gap in~\cite{BR} and prove some uniform error estimates in the diffusion approximation regime, for the schemes~\eqref{eq:scheme-strong-intro} and~\eqref{eq:scheme-weak-intro} applied to~\eqref{eq:SDEintro}. The articles~\cite{FrankGottwald:18} and~\cite{LiAbdulleE:08} illustrate why effective numerical approximation of solutions of SDEs may be more subtle than for deterministic problems. Many other techniques have been introduced to design effective methods for the numerical approximation of multiscale SDE systems, let us mention spectral methods~\cite{AbdullePavliotisVaes:17}, heterogeneous multiscale methods~\cite{ELiuVandenEijnden:05}, projective integration methods~\cite{GivonKevrekidisKupferman:06}, equation-free methods~\cite{KevrekidisAl:03}, parareal methods~\cite{LegollLelievreMyerscoughSamaey:20}, micro-macro acceleration methods~\cite{VandecasteeleZielinskiSamaey:20} for instance. We refer to the monographs~\cite{Gobet,KloedenPlaten,MilsteinTretyakov,
Pages} for general results on numerical methods applied to stochastic differential equations.

This article is organized as follows. Section~\ref{sec:setting} describes the setting, in particular the considered multiscale SDE systems are presented in Section~\ref{sec:setting-SDE} and the numerical schemes studied in this work are given in Section~\ref{sec:setting-schemes}. The main results of this article are stated and discussed in Section~\ref{sec:main}: uniform strong error estimates are given in Theorem~\ref{theo:strong} (Section~\ref{sec:main-strong}) and weak error estimates are given in Theorem~\ref{theo:weak} (Section~\ref{sec:main-weak}). Moment bounds are stated and proved in Section~\ref{sec:moments}. Theorem~\ref{theo:strong} is proved in Section~\ref{sec:proofstrong} whereas Theorem~\ref{theo:weak} is proved in Section~\ref{sec:proofweak}. Auxiliary regularity results for solutions of Kolmogorov equations, with a careful analysis of the dependence with respect to $\epsilon$, are stated in Section~\ref{sec:Kolmogorov} and proved in Section~\ref{sec:proofKolmogorov}.

\section{Setting}\label{sec:setting}

Let $d\in\N$ be an integer. The norm and inner product in the standard Euclidian space $\R^d$ are denoted by $|\cdot|$ and $\langle \cdot,\cdot\rangle$ respectively. The set of $d\times d$ matrices with real-valued entries is denoted by $\mathcal{M}_d(\R)$. The same notation is used to denote the norm and inner product in the space $\R^{2d}$. Let $\bigl(\beta(t)\bigr)_{t\ge 0}$ be a $\R^d$-valued standard Wiener process, defined on a probability space $(\Omega,\mathcal{F},\mathbb{P})$ which satisfies the usual conditions. The expectation operator is denoted by $\E[\cdot]$.

The time-scale separation parameter is denoted by $\epsilon$. Without loss of generality, it is assumed that $\epsilon\in(0,\epsilon_0)$, where $\epsilon_0$ is an arbitrary positive parameter. The time-step size of the numerical schemes is denoted by $\Delta t$. It is assumed that $\Delta t=T/N$ where $T\in(0,\infty)$ is an arbitrary positive real number, and $N\in\N$ is an integer. For all $n\in\{0,\ldots,N\}$, let $t_n=n\Delta t$. Without loss of generality, it is assumed that $\Delta t\in(0,\Delta t_0)$, where $\Delta t_0=T/N_0$ is an arbitrary positive real number. Equivalently, it is assumed that $N\ge N_0$. For all $n\in\{0,\ldots,N-1\}$, set $\Delta\beta_n=\beta(t_{n+1})-\beta(t_n)$.

If $\varphi:\R^d\to\R$ is a mapping of class $\mathcal{C}^3$, its first, second and third order derivatives are denoted by $\nabla\varphi$, $\nabla^2\varphi$ and $\nabla^3\varphi$ respectively. Set
\begin{align*}
\vvvert\varphi\vvvert_1&=\underset{x\in\R^d}\sup~\underset{h^1\in\R^d}\sup~|\nabla \varphi(x).h^1|,\\
\vvvert\varphi\vvvert_2&=\vvvert\varphi\vvvert_1+\underset{x\in\R^d}\sup~\underset{h^1,h^2\in\R^d}\sup~|\nabla^2\varphi(x).(h^1,h^2)|,\\
\vvvert\varphi\vvvert_3&=\vvvert\varphi\vvvert_2+\underset{x\in\R^d}\sup~\underset{h^1,h^2,h^3\in\R^d}\sup~|\nabla^3\varphi(x).(h^1,h^2,h^3)|.
\end{align*}
If $\phi:(q,p)\in\R^d\times\R^d\to\R$ is of class $\mathcal{C}^1$, $\nabla_q\phi$ and $\nabla_p\phi$ are the partial derivatives of $\phi$ with respect to the variables $q$ and $p$ respectively. Similar notation is used for higher order derivatives.

\subsection{The multiscale SDE system}\label{sec:setting-SDE}

We consider the following class of multiscale SDE systems
\begin{equation}\label{eq:SDE}
\left\lbrace
\begin{aligned}
dq^\epsilon(t)&=\frac{p^\epsilon(t)}{\epsilon}dt\\
dp^\epsilon(t)&=-\frac{p^\epsilon(t)}{\epsilon^2}dt+\frac{f(q^\epsilon(t))}{\epsilon}dt+\frac{\sigma(q^\epsilon(t))}{\epsilon}d\beta(t),
\end{aligned}
\right.
\end{equation}
where $q^\epsilon(t)\in\R^d$ and $p^\epsilon(t)\in\R^d$ for all $t\ge 0$. The mappings $f:\R^d\to\R^d$ and $\sigma:\R^d\to \mathcal{M}_d(\R)$ satisfy Assumption~\ref{ass:f-sigma} below.
\begin{ass}\label{ass:f-sigma}
Let $f:\R^d\to\R^d$ and $\sigma:\R^d\to \mathcal{M}_d(\R)$ be mappings of class $\mathcal{C}^3$, with bounded derivatives of order $1,2,3$. The mapping $\sigma$ is assumed to be bounded.
\end{ass}

For all $q\in\R^d$ and $i,j\in\{1,\ldots,d\}$, set
\begin{equation}\label{eq:a}
a_{ij}(q)=\sum_{k=1}^{d}\sigma_{ik}(q)\sigma_{jk}(q).
\end{equation}
Note that $a$ is of class $\mathcal{C}^3$, and since $\sigma$ and its derivatives are bounded, $a$ and its derivatives are bounded. Therefore the mapping $a$ is globally Lipschitz continuous: there exists $C\in(0,\infty)$ such that for all $q_1,q_2\in\R^d$ one has
\begin{equation}\label{eq:aLip}
\sum_{i,j=1}^{d}|a_{ij}(q_2)-a_{ij}(q_1)|\le C|q_2-q_1|.
\end{equation}

The initial values for the SDE system~\eqref{eq:SDE} are given by $q^\epsilon(0)=q_0^\epsilon$ and $p^\epsilon(0)=p_0^\epsilon$, such that Assumption~\ref{ass:init} below is satisfied.
\begin{ass}\label{ass:init}
There exists $q_0^0\in\R^d$ such that
\[
q_0^\epsilon\underset{\epsilon\to 0}\to q_0^0.
\]
Moreover, one has the following uniform upper bound:
\[
\underset{\epsilon\in(0,\epsilon_0)}\sup~|p_0^\epsilon|<\infty.
\]
\end{ass}
It is assumed that the initial values $q_0^\epsilon\in\R^d$ and $q_0^\epsilon\in\R$ are deterministic. The case of random initial values, independent of the Wiener process $\bigl(\beta(t)\bigr)_{t\ge 0}$, can be treated by a standard conditioning argument, provided that suitable moment bounds are satisfied. This treatment is omitted in the sequel.

All the estimates below depend on the value of $\underset{\epsilon\in(0,\epsilon_0)}\sup~|p_0^\epsilon|$, but this is not indicated explicitly.

Under Assumptions~\ref{ass:f-sigma} and~\ref{ass:init}, the SDE system~\eqref{eq:SDE} admits a unique solution $\bigl(q^\epsilon(t),p^\epsilon(t)\bigr)_{t\ge 0}$, since $f$ and $\sigma$ are globally Lipschitz continuous. The solution can be expressed as follows: for all $t\ge 0$, one has
\begin{equation}\label{eq:SDE-mild}
\left\lbrace
\begin{aligned}
q^\epsilon(t)&=q_0^\epsilon+\frac{1}{\epsilon}\int_{0}^{t}p^\epsilon(s)ds,\\
p^\epsilon(t)&=e^{-\frac{t}{\epsilon^2}}p_0^\epsilon+\frac{1}{\epsilon}\int_{0}^{t}e^{-\frac{t-s}{\epsilon^2}}f(q^\epsilon(s))ds+\frac{1}{\epsilon}\int_{0}^{t}e^{-\frac{t-s}{\epsilon^2}}d\beta(s).
\end{aligned}
\right.
\end{equation}

The following change of unknowns is employed below: for all $t\ge 0$, set
\begin{equation}\label{eq:cdvPQ}
\left\lbrace
\begin{aligned}
Q^\epsilon(t)&=q^\epsilon(t)+\epsilon p^\epsilon(t)\\
P^\epsilon(t)&=\epsilon p^\epsilon(t).
\end{aligned}
\right.
\end{equation}
The $\R^{2d}$-valued process $\bigl(Q^\epsilon(t),P^\epsilon(t)\bigr)_{t\ge 0}$ is then solution of the SDE system
\begin{equation}\label{eq:SDE-PQ}
\left\lbrace
\begin{aligned}
dQ^\epsilon(t)&=f\bigl(Q^\epsilon(t)-P^\epsilon(t)\bigr)dt+\sigma\bigl(Q^\epsilon(t)-P^\epsilon(t)\bigr)d\beta(t),\\
dP^\epsilon(t)&=-\frac{P^\epsilon(t)}{\epsilon^2}dt+f\bigl(Q^\epsilon(t)-P^\epsilon(t)\bigr)dt+\sigma\bigl(Q^\epsilon(t)-P^\epsilon(t)\bigr)d\beta(t).
\end{aligned}
\right.
\end{equation}
To retrieve properties of $q^\epsilon(t)$, note that for all $t\ge 0$ one has
\[
q^\epsilon(t)=Q^\epsilon(t)-P^\epsilon(t).
\]
The change of unknowns is instrumental in the proofs of moment bounds and strong error estimates for $Q^\epsilon$ and $P^\epsilon$, however weak error analysis is performed using only the unknowns $q^\epsilon$ and $p^\epsilon$.

\subsection{The numerical schemes}\label{sec:setting-schemes}

We introduce two numerical schemes: the first one is used to obtain strong approximation of $q^\epsilon(t)$, whereas the second one is used to obtain weak approximation of $q^\epsilon(t)$. The same notation is used for the two schemes, since it will always be clear in the statements of the results and in the analysis below which scheme is considered.

The first numerical scheme is defined as follows:
\begin{equation}\label{eq:scheme-strong}
\left\lbrace
\begin{aligned}
q_{n+1}^{\epsilon,\Delta t}&=q_n^{\epsilon,\Delta t}+\frac{\Delta t}{\epsilon}p_{n+1}^{\epsilon,\Delta t}\\
p_{n+1}^{\epsilon,\Delta t}&=p_n^{\epsilon,\Delta t}-\frac{\Delta t}{\epsilon^2}p_{n+1}^{\epsilon,\Delta t}+\frac{\Delta tf(q_n^{\epsilon,\Delta t})}{\epsilon}+\frac{\sigma(q_n^{\epsilon,\Delta t})}{\epsilon}\Delta\beta_n
\end{aligned}
\right.
\end{equation}
with the initial values $q_0^{\epsilon,\Delta t}=q_0^\epsilon$ and $p_0^{\epsilon,\Delta t}=p_0^\epsilon$ (given by Assumption~\ref{ass:init}). In the scheme~\eqref{eq:scheme-strong}, the $q$-component is discretized explicitly, whereas the $p$-component is treated implicitly. This choice is made to ensure stability properties (in particular to be able to choose the time-step size $\Delta t$ independently of the time scale separation parameter $\epsilon$) in the second equation, and to ensure good behavior of $q_n^{\epsilon,\Delta t}$ when $\epsilon$ vanishes, in the first equation, as will be explained below. In addition, discretizing the $q$-component explicitly is needed to have consistent approximation in the sense of It\^o of the contribution of the noise. Note that, in fact, the scheme~\eqref{eq:scheme-strong} can be implemented explicitly in practice owing to the following equivalent formulation, computing first $p_{n+1}^{\epsilon,\Delta t}$ and then $q_{n+1}^{\epsilon,\Delta t}$:
\[
\left\lbrace
\begin{aligned}
q_{n+1}^{\epsilon,\Delta t}&=q_n^{\epsilon,\Delta t}+\frac{\Delta t}{\epsilon}p_{n+1}^{\epsilon,\Delta t}\\
p_{n+1}^{\epsilon,\Delta t}&=\frac{1}{1+\frac{\Delta t}{\epsilon^2}}\Bigl(p_n^{\epsilon,\Delta t}+\frac{\Delta tf(q_n^{\epsilon,\Delta t})}{\epsilon}+\frac{\sigma(q_n^{\epsilon,\Delta t})}{\epsilon}\Delta\beta_n\Bigr).
\end{aligned}
\right.
\]

The second numerical scheme is defined as follows:
\begin{equation}\label{eq:scheme-weak}
\left\lbrace
\begin{aligned}
q_{n+1}^{\epsilon,\Delta t}&=q_n^{\epsilon,\Delta t}+\epsilon\bigl(1-e^{-\frac{\Delta t}{\epsilon^2}}\bigr)p_n^{\epsilon,\Delta t}+\bigl(\Delta t-\epsilon^2(1-e^{-\frac{\Delta t}{\epsilon^2}})\bigr)f(q_n^{\epsilon,\Delta t})\\
&+\sigma(q_n^{\epsilon,\Delta t})(\beta(t_{n+1})-\beta(t_n))-\sigma(q_n^{\epsilon,\Delta t})\int_{t_n}^{t_{n+1}}e^{-\frac{t_{n+1}-s}{\epsilon^2}}d\beta(s)\\
p_{n+1}^{\epsilon,\Delta t}&=e^{-\frac{\Delta t}{\epsilon^2}}p_n^{\epsilon,\Delta t}+\epsilon(1-e^{-\frac{\Delta t}{\epsilon^2}})f(q_n^{\epsilon,\Delta t})+\frac{1}{\epsilon}\sigma(q_n^{\epsilon,\Delta t})\int_{t_n}^{t_{n+1}}e^{-\frac{t_{n+1}-s}{\epsilon^2}}d\beta(s).
\end{aligned}
\right.
\end{equation}
The numerical scheme~\eqref{eq:scheme-weak} is appropriate to obtain approximation in distribution of $q^\epsilon(t_n)$ and $p^\epsilon(t_n)$: indeed, it suffices to sample at each iteration a $\R^{2d}$-valued centered Gaussian random variable
\[
\Bigl(\beta(t_{n+1})-\beta(t_n),\int_{t_n}^{t_{n+1}}e^{-\frac{t_{n+1}-s}{\epsilon^2}}d\beta(s)\Bigr),
\]
or equivalently of a family of $d$ independent $\R^2$-valued centered Gaussian random variables
\[
\Bigl(\beta_j(t_{n+1})-\beta_j(t_n),\int_{t_n}^{t_{n+1}}e^{-\frac{t_{n+1}-s}{\epsilon^2}}d\beta_j(s)\Bigr),
\]
with $j=1,\ldots,d$, which have the same covariance matrix with entries given by
\begin{align*}
\E[(\beta_j(t_{n+1})-\beta_j(t_n))^2]&=\Delta t\\
\E[(\beta_j(t_{n+1})-\beta_j(t_n))\int_{t_n}^{t_{n+1}}e^{-\frac{t_{n+1}-s}{\epsilon^2}}d\beta_j(s)]&=\int_{t_n}^{t_{n+1}}e^{-\frac{t_{n+1}-s}{\epsilon^2}}ds=\epsilon^2(1-e^{-\frac{\Delta t}{\epsilon^2}})\\
\E[\bigl(\int_{t_n}^{t_{n+1}}e^{-\frac{t_{n+1}-s}{\epsilon^2}}d\beta(s)\bigr)^2]&=\int_{t_n}^{t_{n+1}}e^{-2\frac{t_{n+1}-s}{\epsilon^2}}ds=\frac{\epsilon^2}{2}(1-e^{-\frac{2\Delta t}{\epsilon^2}}).
\end{align*}
It suffices to compute the square root or the Cholesky decomposition of the covariance matrix to sample the required Gaussian random variables. The construction of the second numerical scheme~\eqref{eq:scheme-weak} is motivated by the following result.
\begin{propo}\label{propo:SDEtilde}
Let $\epsilon\in(0,\epsilon_0)$ and $\Delta t\in(0,\Delta t_0)$. Introduce the continuous-time auxiliary process $\bigl(\tilde{q}^{\epsilon,\Delta t}(t),\tilde{p}^{\epsilon,\Delta t}(t)\bigr)_{t\ge 0}$ defined such that for all $n\in\{0,\ldots,N-1\}$ and $t\in[t_n,t_{n+1}]$, one has
\begin{equation}\label{eq:SDEtilde}
\left\lbrace
\begin{aligned}
d\tilde{q}^{\epsilon,\Delta t}(t)&=\frac{\tilde{p}^{\epsilon,\Delta t}(t)}{\epsilon}dt\\
d\tilde{p}^{\epsilon,\Delta t}(t)&=-\frac{\tilde{p}^{\epsilon,\Delta t}(t)}{\epsilon^2}dt+\frac{f(\tilde{q}^{\epsilon,\Delta t}(t_n))}{\epsilon}dt+\frac{\sigma(\tilde{q}^{\epsilon,\Delta t}(t_n))}{\epsilon}d\beta(t),
\end{aligned}
\right.
\end{equation}
with initial values $\tilde{q}^{\epsilon,\Delta t}(0)=q_0^\epsilon$ and $\tilde{p}^{\epsilon,\Delta t}(0)=p_0^\epsilon$, and such that $t\in[0,T]\mapsto \tilde{q}^{\epsilon,\Delta t}(t)$ and $t\in[0,T]\mapsto \tilde{p}^{\epsilon,\Delta t}(t)$ are continuous. Then for all $n\in\{1,\ldots,N\}$ one has
\[
\bigl(q_n^{\epsilon,\Delta t},p_n^{\epsilon,\Delta t}\bigr)=\bigl(\tilde{q}^{\epsilon,\Delta t}(t_n),\tilde{p}^{\epsilon,\Delta t}(t_n)\bigr).
\]
\end{propo}

\begin{proof}[Proof of Proposition~\ref{propo:SDEtilde}]
Let $n\in\{0,\ldots,N-1\}$, then for all $t\in[t_n,t_{n+1}]$, one has
\[
\tilde{p}^{\epsilon,\Delta t}(t)=e^{-\frac{t-t_n}{\epsilon^2}}\tilde{p}^{\epsilon,\Delta t}(t_n)+\epsilon(1-e^{-\frac{t-t_n}{\epsilon^2}})f(\tilde{q}^{\epsilon,\Delta t}(t_n))+\frac{1}{\epsilon}\sigma(\tilde{q}^{\epsilon,\Delta t}(t_n))\int_{t_n}^{t}e^{-\frac{t-s}{\epsilon^2}}d\beta(s).
\]
This then gives the equality
\begin{align*}
\tilde{q}^{\epsilon,\Delta t}(t_{n+1})&=\tilde{q}^{\epsilon,\Delta t}(t_n)+\frac{1}{\epsilon}\int_{t_n}^{t_{n+1}}\tilde{p}^{\epsilon,\Delta t}(t)dt\\
&=\tilde{q}^{\epsilon,\Delta t}(t_n)+\epsilon(1-e^{-\frac{\Delta t}{\epsilon^2}})\tilde{p}^{\epsilon,\Delta t}(t_n)+\bigl(\Delta t-\epsilon^2(1-e^{-\frac{\Delta t}{\epsilon^2}})\bigr)f(\tilde{q}^{\epsilon,\Delta t}(t_n))\\
&+\frac{1}{\epsilon^2}\sigma(\tilde{q}^{\epsilon,\Delta t}(t_n))\int_{t_n}^{t_{n+1}}\int_{t_n}^{t}e^{-\frac{t-s}{\epsilon^2}}d\beta(s)dt,
\end{align*}
where using the stochastic Fubini theorem one obtains
\begin{align*}
\int_{t_n}^{t_{n+1}}\int_{t_n}^{t}e^{-\frac{t-s}{\epsilon^2}}d\beta(s)dt&=\int_{t_n}^{t_{n+1}}\int_{s}^{t_{n+1}}e^{-\frac{t-s}{\epsilon^2}}dtd\beta(s)\\
&=\int_{t_n}^{t_{n+1}}\epsilon^2(1-e^{-\frac{t_{n+1}-s}{\epsilon^2}})d\beta(s)\\
&=\epsilon^2(\beta(t_{n+1})-\beta(t_n))-\epsilon^2\int_{t_n}^{t_{n+1}}e^{-\frac{t_{n+1}-s}{\epsilon^2}}d\beta(s).
\end{align*}
Since $\tilde{q}^{\epsilon,\Delta t}(0)=q_0^{\epsilon,\Delta t}$ and $\tilde{p}^{\epsilon,\Delta t}(0)=p_0^{\epsilon,\Delta t}$, it is then straightforward to check that $\tilde{q}^{\epsilon,\Delta t}(t_n)=q_n^{\epsilon,\Delta t}$ and $\tilde{p}^{\epsilon,\Delta t}(t_n)=p_n^{\epsilon,\Delta t}$ for all $n\in\{1,\ldots,N\}$. The proof of Proposition~\ref{propo:SDEtilde} is thus completed.
\end{proof}

As a consequence of Proposition~\ref{propo:SDEtilde}, observe that if $f$ and $\sigma$ are constant, then the scheme~\eqref{eq:scheme-weak} is exact: $q_n^{\epsilon,\Delta t}=q^\epsilon(t_n)$ and $p_n^{\epsilon,\Delta t}=p^\epsilon(t_n)$. Using the implementation of the scheme explained above, this means that one obtains a scheme which is exact in distribution. The numerical scheme~\eqref{eq:scheme-weak} is thus constructed by freezing the values of the $q$-component on each interval $[t_n,t_{n+1}]$ when applying the mappings $f$ and $\sigma$, and by computing the exact solution of the SDE depending on $f(q_n^{\epsilon,\Delta t})$ and $\sigma(q_n^{\epsilon,\Delta t})$ on each interval.

Note that the continuous auxiliary processes $\bigl(\tilde{q}^{\epsilon,\Delta t}(t),\tilde{p}^{\epsilon,\Delta t}(t)\bigr)_{t\ge 0}$ play a role below in the proof of the uniform weak error estimates for the numerical scheme~\eqref{eq:scheme-weak}.

Like in the continuous-time setting (see~\eqref{eq:cdvPQ}), it is convenient to introduce auxiliary unknowns
\begin{equation}\label{eq:cdvPQscheme}
\left\lbrace
\begin{aligned}
Q_n^{\epsilon,\Delta t}&=q_n^{\epsilon,\Delta t}+\epsilon p_n^{\epsilon,\Delta t}\\
P_n^{\epsilon,\Delta t}&=\epsilon p_n^{\epsilon,\Delta t}.
\end{aligned}
\right.
\end{equation}
After proving some properties for the unknowns $Q_n^{\epsilon,\Delta t}$ and $P_n^{\epsilon,\Delta t}$, the identity $q_n^{\epsilon,\Delta t}=Q_n^{\epsilon,\Delta t}-P_n^{\epsilon,\Delta t}$ is then used to retrieve properties of the unknown $q_n^{\epsilon,\Delta t}$.

If the first numerical scheme~\eqref{eq:scheme-strong} is used, the system after the change of variables reads
\begin{equation}\label{eq:scheme-strong-QP}
\left\lbrace
\begin{aligned}
Q_{n+1}^{\epsilon,\Delta t}&=Q_n^{\epsilon,\Delta t}+\Delta tf(Q_n^{\epsilon,\Delta t}-P_n^{\epsilon,\Delta t})+\sigma(Q_n^{\epsilon,\Delta t}-P_n^{\epsilon,\Delta t})\Delta\beta_n\\
P_{n+1}^{\epsilon,\Delta t}&=\frac{1}{1+\frac{\Delta t}{\epsilon^2}}\Bigl(P_n^{\epsilon,\Delta t}+\Delta tf(Q_n^{\epsilon,\Delta t}-P_n^{\epsilon,\Delta t})+\sigma(Q_n^{\epsilon,\Delta t}-P_n^{\epsilon,\Delta t})\Delta\beta_n\Bigr).
\end{aligned}
\right.
\end{equation}
If the second numerical scheme~\eqref{eq:scheme-weak} is used, the system after the change of variables reads
\begin{equation}\label{eq:scheme-weak-QP}
\left\lbrace
\begin{aligned}
Q_{n+1}^{\epsilon,\Delta t}&=Q_n^{\epsilon,\Delta t}+\Delta tf(Q_n^{\epsilon,\Delta t}-P_n^{\epsilon,\Delta t})+\sigma(Q_n^{\epsilon,\Delta t}-P_n^{\epsilon,\Delta t})\Delta\beta_n\\
P_{n+1}^{\epsilon,\Delta t}&=e^{-\frac{\Delta t}{\epsilon^2}}P_n^{\epsilon,\Delta t}+\epsilon^2(1-e^{-\frac{\Delta t}{\epsilon^2}})f(Q_n^{\epsilon,\Delta t}-P_n^{\epsilon,\Delta t})+\sigma(Q_n^{\epsilon,\Delta t}-P_n^{\epsilon,\Delta t})\int_{t_n}^{t_{n+1}}e^{-\frac{t_{n+1}-s}{\epsilon^2}}d\beta(s).
\end{aligned}
\right.
\end{equation}

It is straightforward to check that the numerical schemes~\eqref{eq:scheme-strong} and~\eqref{eq:scheme-weak} give consistent strong and weak approximation respectively of the solution of the system~\eqref{eq:SDE} when $\Delta t\to 0$, for any fixed value of the time-scale separation parameter $\epsilon\in(0,\epsilon_0)$. Since the objective of this article is to prove that these schemes can be run with a cost independent of $\epsilon$, it is relevant to study the behavior when $\epsilon\to 0$ of $q^\epsilon(t)$ and of $q_n^{\epsilon,\Delta t}$.

\subsection{Asymptotic behavior when the time scale separation parameter vanishes}

Consider the stochastic differential equation
\begin{equation}\label{eq:limitingSDE}
dq^0(t)=f(q^0(t))dt+\sigma(q^0(t))d\beta(t),
\end{equation}
where $q^0(t)\in\R^d$, with initial value $q^0(0)=q_0^0=\underset{\epsilon\to 0}\lim~q_0^\epsilon$ (see Assumption~\ref{ass:init}). Owing to Assumption~\ref{ass:f-sigma}, $f$ and $\sigma$ are globally Lipschitz continuous, therefore the SDE~\eqref{eq:limitingSDE} admits a unique solution $\bigl(q^0(t)\bigr)_{t\ge 0}$. This solution satisfies the identity
\begin{equation}\label{eq:limitingSDEintegral}
q^0(t)=q_0^0+\int_0^t f(q^0(s))ds+\int_0^t \sigma(q^0(s))d\beta(s)
\end{equation}
for all $t\ge 0$.

Consider also the standard Euler--Maruyama scheme applied to the SDE~\eqref{eq:limitingSDE}, with time-step size $\Delta t$: set $q_0^{0,\Delta t}=q_0^0$ and for all $n\in\{0,\ldots,N-1\}$, set
\begin{equation}\label{eq:limitingscheme}
q_{n+1}^{0,\Delta t}=q_n^{0,\Delta t}+\Delta tf(q_n^{0,\Delta t})+\sigma(q_n^{0,\Delta t})\Delta\beta_n,
\end{equation}
where we recall that $\Delta\beta_n=\beta(t_{n+1})-\beta(t_n)$.

One has the following convergence result when $\epsilon\to 0$.
\begin{propo}\label{propo:cv}
Let Assumptions~\ref{ass:f-sigma} and~\ref{ass:init} be satisfied. For all $T\in(0,\infty)$, there exists $C(T)\in(0,\infty)$ such that for all $\epsilon\in(0,\epsilon_0)$, one has
\begin{equation}\label{eq:cvSDE}
\underset{0\le t\le T}\sup~\E[|q^\epsilon(t)-q^0(t)|^2]\le C(T)\Bigl(|q_0^\epsilon-q_0^0|^2+\epsilon^2\bigl(1+|q_0^\epsilon|^2+|p_0^\epsilon|^2\bigr)\Bigr)\underset{\epsilon\to 0}\to 0.
\end{equation}
In addition, for all $\Delta t\in(0,\Delta t_0)$, one has
\begin{equation}\label{eq:cvscheme}
\underset{n=0,\ldots,N}\sup~\E[|q_n^{\epsilon,\Delta t}-q_n^{0,\Delta t}|^2]\le C(T)\Bigl(|q_0^\epsilon-q_0^0|^2+\epsilon^2\bigl(1+|q_0^\epsilon|^2+|p_0^\epsilon|^2\bigr)\Bigr)\underset{\epsilon\to 0}\to 0,
\end{equation}
where $\bigl(q_n^{\epsilon,\Delta t},p_n^{\epsilon,\Delta t}\bigr)_{n=0,\ldots,N}$ is given either by~\eqref{eq:scheme-strong} or by~\eqref{eq:scheme-weak}.
\end{propo}
The proof of Proposition~\ref{propo:cv} is postponed to Section~\ref{sec:moments-cv}, since it requires moment bounds (uniform with respect to $\epsilon$) which are stated and proved in Section~\ref{sec:moments}.

\section{Main results}\label{sec:main}

We are now in position to state the main results of this article. First, in Section~\ref{sec:main-strong}, we study strong error estimates when the numerical scheme~\eqref{eq:scheme-strong} is used. Second, in Section~\ref{sec:main-weak}, we study strong error estimates when the numerical scheme~\eqref{eq:scheme-weak} is used.

\subsection{Uniform strong error estimates}\label{sec:main-strong}

In this subsection, let us consider the numerical scheme~\eqref{eq:scheme-strong}. One has the following result concerning the strong error $\E[|q_n^{\epsilon,\Delta t}-q^\epsilon(n\Delta t)|^2]$ for the $q$-component, when $\Delta t\to 0$.

\begin{theo}\label{theo:strong}
Let Assumptions~\ref{ass:f-sigma} and~\ref{ass:init} be satisfied, and let $\bigl(q_n^{\epsilon,\Delta t}\bigr)_{n\ge 0}$ be given by the numerical scheme~\eqref{eq:scheme-strong}. For all $T\in(0,\infty)$, there exists $C(T)\in(0,\infty)$ such that for all $\Delta t=T/N\in(0,\Delta t_0)$ and $n\in\{1,\ldots,N\}$, one has
\begin{equation}\label{eq:theo-strong}
\underset{\epsilon\in(0,\epsilon_0)}\sup~\E[|q_n^{\epsilon,\Delta t}-q^\epsilon(n\Delta t)|^2]\le C(T)(1+|q_0^0|^2)\Delta t+\frac{C(T)}{(n+1)^2}.
\end{equation}
\end{theo}
The remarkable property of~\eqref{eq:theo-strong} is that the error estimate is uniform with respect to the time-scale separation parameter $\epsilon$. In fact, the proof provides a more precise error estimate, which is combined with Assumption~\ref{ass:init} to obtain~\eqref{eq:theo-strong}: one has
\begin{equation}\label{eq:theo-strong-precis}
\E[|q_n^{\epsilon,\Delta t}-q^\epsilon(n\Delta t)|^2]\le C(T)(1+|q_0^\epsilon|^2+|p_0^\epsilon|^2)\Delta t+\frac{C(T)\epsilon^2}{(n+1)^2}|p_0^\epsilon|^2.
\end{equation}

As a consequence, one obtains the following uniform strong error estimates. First, choosing $N=n$, and using the fact that $1/N^2=\Delta t^2/T^2$, one obtains the uniform strong error estimate at the final time $T=N\Delta t$:
\begin{equation}\label{eq:theo-strong-final}
\underset{\epsilon\in(0,\epsilon_0)}\sup~\E[|q_N^{\epsilon,\Delta t}-q^\epsilon(N\Delta t)|^2]\le C(T)(1+|q_0^0|^2)\Delta t.
\end{equation}
Second, if $p_0^\epsilon=0$, one also obtains
\[
\underset{\epsilon\in(0,\epsilon_0)}\sup~\underset{n=0,\ldots,N}\sup~\E[|q_n^{\epsilon,\Delta t}-q^\epsilon(n\Delta t)|^2]\le C(T)(1+|q_0^0|^2)\Delta t,
\]
where the error estimate is uniform with respect to both $\epsilon\in(0,\epsilon_0)$ and $n\in\{0,\ldots,N\}$.

Observe that by letting $\epsilon\to 0$ in the strong error estimate~\eqref{eq:theo-strong-precis} and using Proposition~\ref{propo:cv}, one retrieves the standard strong error estimate for the Euler--Maruyama scheme~\eqref{eq:limitingscheme} applied to the SDE~\eqref{eq:limitingSDE}: one has
\begin{equation}\label{eq:theo-strong-limiting}
\underset{n=0,\ldots,N}\sup~\E[|q_n^{0,\Delta t}-q^0(n\Delta t)|^2]\le C(T)(1+|q_0^0|^2)\Delta t.
\end{equation}
In general, when $\sigma$ is not constant, the error estimate~\eqref{eq:theo-strong-limiting} above is optimal: the Euler--Maruyama scheme has strong order of convergence equal to $1/2$, thus the order of convergence in Theorem~\ref{theo:strong} is also optimal in general. If the diffusion coefficient $\sigma$ is constant, the strong order of convergence of the Euler--Maruyama scheme applied to the SDE~\eqref{eq:limitingSDE} driven by additive noise is in fact equal to $1$: one can replace $\Delta t$ by $\Delta t^2$ in the right-hand side of~\eqref{eq:theo-strong-limiting}. However, it does not seem possible to improve the order of convergence in Theorem~\ref{theo:strong} when $\sigma$ is assumed to be constant using the arguments of the proof: in that case Theorem~\ref{theo:strong} may not be optimal.

Let us also provide a strong error estimate for the second numerical scheme~\eqref{eq:scheme-weak}.
\begin{propo}\label{propo:strong}
Let Assumptions~\ref{ass:f-sigma} and~\ref{ass:init} be satisfied, and let $\bigl(q_n^{\epsilon,\Delta t}\bigr)_{n\ge 0}$ be given by the numerical scheme~\eqref{eq:scheme-weak}. For all $T\in(0,\infty)$, there exists $C(T)\in(0,\infty)$ such that for all $\Delta t=T/N\in(0,\Delta t_0)$, one has
\begin{equation}\label{eq:propo-strong}
\underset{\epsilon\in(0,\epsilon_0)}\sup~\underset{n=0,\ldots,N}\sup~\E[|q_n^{\epsilon,\Delta t}-q^\epsilon(n\Delta t)|^2]\le C(T)(1+|q_0^0|^2)\Delta t.
\end{equation}
\end{propo}
The proof of Proposition~\ref{propo:strong} would follow the same strategy as the proof of Theorem~\ref{theo:strong}. In fact, compared with the proof of Theorem~\ref{theo:strong} several of the error terms vanish, which is due to replacing $\frac{1}{1+\frac{\Delta t}{\epsilon^2}}$ by $e^{-\frac{\Delta t}{\epsilon^2}}$ in the expressions. This explains why the extra term in~\eqref{eq:theo-strong-precis} does not appear in~\eqref{eq:propo-strong}. The details are omitted.

\subsection{Uniform weak error estimates}\label{sec:main-weak}

In this subsection, let us consider the numerical scheme~\eqref{eq:scheme-weak}. One has the following result concerning the weak error $|\E[\varphi(q_n^{\epsilon,\Delta t})]-\E[\varphi(q^\epsilon(n\Delta t))]|$ for the $q$-component, when $\Delta t\to 0$.

\begin{theo}\label{theo:weak}
Let Assumptions~\ref{ass:f-sigma} and~\ref{ass:init} be satisfied, and let $\bigl(q_n^{\epsilon,\Delta t}\bigr)_{n\ge 0}$ be given by the numerical scheme~\eqref{eq:scheme-weak}. For all $T\in(0,\infty)$, there exists $C(T)\in(0,\infty)$ such that for all $\Delta t=T/N\in(0,\Delta t_0)$, any function $\varphi:\R^d\to \R$ of class $\mathcal{C}^3$ with bounded derivatives of order $1,2,3$, and all $\epsilon\in(0,\epsilon_0)$, one has
\begin{equation}\label{eq:theo-weak}
\underset{n=0,\ldots,N}\sup~\big|\E[\varphi(q_n^{\epsilon,\Delta t})]-\E[\varphi(q^\epsilon(n\Delta t))]\big|\le C(T)\vvvert\varphi\vvvert_3(1+|q_0^0|^2)\bigl(\Delta t+\mathcal{R}(\epsilon,\Delta t)\bigr),
\end{equation}
where the residual error term $\mathcal{R}(\epsilon,\Delta t)$ is defined by
\begin{equation}\label{eq:R}
\mathcal{R}(\epsilon,\Delta t)=\int_{0}^{\Delta t}\frac{\epsilon}{\Delta t}\bigl(1-e^{-\frac{t}{\epsilon^2}}\bigr)dt\ge 0
\end{equation}
and satisfies the following inequality: there exists $C\in(0,\infty)$ such that for all $\Delta t\in(0,\Delta t_0)$ and $\epsilon\in(0,\epsilon_0)$ one has
\begin{equation}\label{eq:Rineq}
\mathcal{R}(\epsilon,\Delta t)\le C\min\bigl(\epsilon,\Delta t^{\frac12},\frac{\Delta t}{\epsilon}\bigr).
\end{equation}
\end{theo}
The proof of the inequality~\eqref{eq:Rineq} is straightforward: indeed, one has the inequalities $0\le 1-e^{-s}\le 2$ and $0\le 1-e^{-s} \le s$, and therefore $0\le 1-e^{-s}\le (2s)^{\frac12}$, for all $s\ge 0$.  Observe that one obtains the uniform upper bound
\[
\underset{\epsilon\in(0,\epsilon_0)}\sup~\mathcal{R}(\epsilon,\Delta t)\le C\Delta t^{\frac12}
\]
which is optimal: indeed, choosing $\epsilon=\sqrt{\Delta t}$, one has
\[
\frac{1}{\Delta t^{\frac12}}\mathcal{R}(\sqrt{\Delta t},\Delta t)=\frac{1}{\Delta t}\int_{0}^{\Delta t}(1-e^{-\frac{t}{\Delta t}})dt=\int_0^1(1-e^{-s})ds=e^{-1}
\]
for all $\Delta t\in(0,\Delta t_0)$.

Let us state two immediate consequences of Theorem~\eqref{theo:weak}. On the one hand, one obtains the following uniform weak error estimate
\begin{equation}\label{eq:theo-weak-unif}
\underset{\epsilon\in(0,\epsilon_0)}\sup~\underset{n=0,\ldots,N}\sup~\big|\E[\varphi(q_n^{\epsilon,\Delta t})]-\E[\varphi(q^\epsilon(n\Delta t))]\big|\le C(T)\vvvert\varphi\vvvert_3(1+|q_0^0|^2)\Delta t^{\frac12}
\end{equation}
where the order of convergence is equal to $1/2$ (and this cannot be improved when using~\eqref{eq:theo-weak} owing to the observation above). On the other hand, letting $\epsilon\to 0$, one retrieves the standard weak error estimate with order $1$ for the Euler--Maruyama scheme~\eqref{eq:limitingscheme} applied to the SDE~\eqref{eq:limitingSDE}: one has
\begin{equation}\label{eq:theo-weak-limiting}
\underset{n=0,\ldots,N}\sup~\big|\E[\varphi(q_n^{0,\Delta t})]-\E[\varphi(q^0(n\Delta t))]\big|\le C(T)\vvvert\varphi\vvvert_3(1+|q_0^0|^2)\Delta t.
\end{equation}
For a fixed value of $\epsilon$, the weak error estimate~\eqref{eq:theo-weak} also gives order of convergence $1$ with respect $\Delta t$, but the corresponding error estimate is not uniform with respect to $\epsilon$.

Note that it would suffice to apply~\eqref{eq:propo-strong} from Proposition~\ref{propo:strong} above to obtain the uniform weak error estimate~\eqref{eq:theo-weak-unif}, but to retrieve~\eqref{eq:theo-weak-limiting} one needs the refined error analysis which gives~\eqref{eq:theo-weak}. Even if the order of convergence in the uniform weak error estimate~\eqref{eq:theo-weak-unif} is $1/2$, one has the following error estimate
\begin{equation}\label{eq:theo-weak-bis}
\underset{n=0,\ldots,N}\sup~\big|\E[\varphi(q_n^{\epsilon,\Delta t})]-\E[\varphi(q^\epsilon(n\Delta t))]\big|\le C(T)\vvvert\varphi\vvvert_3(1+|q_0^0|^2)\bigl(\Delta t+\epsilon\bigr),
\end{equation}
owing to~\eqref{eq:theo-weak} and~\eqref{eq:Rineq}, which is relevant in situations where $\epsilon$ is negligible compared with $\Delta t$.

\section{Moment bounds}\label{sec:moments}

This section is devoted to state and prove moment bounds, which are uniform with respect to $\epsilon\in(0,\epsilon_0)$ and $\Delta t\in(0,\Delta t_0)$, for the solutions of~\eqref{eq:SDE}, of~\eqref{eq:scheme-strong} and of~\eqref{eq:scheme-weak}. The proofs are based on the changes of unknowns introduced in Sections~\eqref{sec:setting-SDE} and~\eqref{sec:setting-schemes}. Even if the arguments are mostly elementary, it is worth giving full details for completeness. The proofs of Propositions~\ref{propo:moment-scheme-strong} and~\ref{propo:moment-scheme-weak} are variants of the proof of Proposition~\ref{propo:moment-SDE} in discrete-time situations. After proving the moment bounds, the proof of Proposition~\ref{propo:cv} is provided.

\subsection{Moment bounds for the SDE system}

\begin{propo}\label{propo:moment-SDE}
Let Assumption~\ref{ass:f-sigma} be satisfied. For all $T\in(0,\infty)$, there exists $C(T)\in(0,\infty)$ such that for all $\epsilon\in(0,\epsilon_0)$, one has
\begin{equation}\label{eq:moment-SDE}
\underset{0\le t\le T}\sup~\Bigl(\E[|q^\epsilon(t)|^2]+\E[|p^\epsilon(t)|^2]\Bigr)\le C(T)\bigl(1+|q_0^\epsilon|^2+|p_0^\epsilon|^2\bigr).
\end{equation}
\end{propo}

\begin{proof}[Proof of Proposition~\ref{propo:moment-SDE}]
Let $\bigl(Q^\epsilon(t),P^\epsilon(t)\bigr)_{t\ge 0}$ be the solution of the SDE system~\eqref{eq:SDE-PQ}. For all $t\in[0,T]$, one has
\begin{equation}\label{eq:mildQP}
\left\lbrace
\begin{aligned}
Q^\epsilon(t)&=Q^\epsilon(0)+\int_0^t f\bigl(Q^\epsilon(s)-P^\epsilon(s)\bigr)ds+\int_0^t \sigma\bigl(Q^\epsilon(s)-P^\epsilon(s)\bigr)d\beta(s)\\
P^\epsilon(t)&=e^{-\frac{t}{\epsilon^2}}P^\epsilon(0)+\int_0^t e^{-\frac{t-s}{\epsilon^2}}f\bigl(Q^\epsilon(s)-P^\epsilon(s)\bigr)ds+\int_0^te^{-\frac{t-s}{\epsilon^2}}\sigma\bigl(Q^\epsilon(s)-P^\epsilon(s)\bigr)d\beta(s).
\end{aligned}
\right.
\end{equation}
Using It\^o's isometry formula, and the Lipschitz continuity property of the mappings $f$ and $\sigma$ (Assumption~\ref{ass:f-sigma}) one obtains the inequality
\[
\E[|Q^\epsilon(t)|^2]+\E[|P^\epsilon(t)|^2]\le C(T)\Bigl(1+\E[|Q^\epsilon(0)|^2]+\E[|P^\epsilon(0)|^2+\int_0^t\bigl(\E[|Q^\epsilon(s)|^2]+\E[|P^\epsilon(s)|^2]\bigr)ds\Bigr).
\]
for all $t\in[0,T]$. Since $Q^\epsilon(0)=q_0^\epsilon+\epsilon p_0^\epsilon$ and $P^\epsilon(0)=\epsilon p_0^\epsilon$, applying Gronwall's lemma yields the moment bound
\[
\underset{0\le t\le T}\sup~\bigl(\E[|Q^\epsilon(t)|^2]+\E[|P^\epsilon(t)|^2]\bigr)\le C(T)\Bigl(1+|q_0^\epsilon|^2+\|p_0^\epsilon|^2\Bigr).
\]
Using the identity $q^\epsilon(t)=Q^\epsilon(t)-P^\epsilon(t)$ then gives
\[
\underset{0\le t\le T}\sup~\E[|q^\epsilon(t)|^2]\le C(T)\Bigl(1+|q_0^\epsilon|^2+|p_0^\epsilon|^2\Bigr).
\]
Using the second identity from~\eqref{eq:SDE-mild} and writing $Q^\epsilon(t)-P^\epsilon(t)=q^\epsilon(t)$,  applying It\^o's isometry formula, and using the Lipschitz continuity properties of $f$ and $\sigma$ and the moment bounds for $q^\epsilon(t)$ obtained above, one then obtains, for all $t\in[0,T]$,
\begin{align*}
\E[|p^\epsilon(t)|^2]&\le C|p_0^\epsilon|^2+\frac{CT}{\epsilon^2}\int_{0}^{t}e^{-\frac{2(t-s)}{\epsilon^2}}\bigl(1+\E[|q^\epsilon(s)|^2]\bigr)ds\\
&\le C(T)\Bigl(1+|q_0^\epsilon|^2+|p_0^\epsilon|^2\Bigr)\Bigl(1+\frac{1}{\epsilon^2}\int_{0}^{t}e^{-\frac{2(t-s)}{\epsilon^2}}ds\Bigr)\\
&\le C(T)\Bigl(1+|q_0^\epsilon|^2+|p_0^\epsilon|^2\Bigr),
\end{align*}
which gives the required moment bounds for $p^\epsilon(t)$, using the identity
\[
\frac{1}{\epsilon^2}\int_0^\infty e^{-\frac{2r}{\epsilon^2}}dr=\frac12.
\]
This concludes the proof of Proposition~\ref{propo:moment-SDE}.
\end{proof}

\subsection{Moment bounds for the first numerical scheme}

\begin{propo}\label{propo:moment-scheme-strong}
Let $\bigl(q_n^{\epsilon,\Delta t},p_n^{\epsilon,\Delta t}\bigr)_{n=0,\ldots,N}$ be given by the numerical scheme~\eqref{eq:scheme-strong}. Let Assumption~\ref{ass:f-sigma} be satisfied. For all $T\in(0,\infty)$, there exists $C(T)\in(0,\infty)$ such that for all $\epsilon\in(0,\epsilon_0)$ and $\Delta t\in(0,\Delta t_0)$, one has
\begin{equation}\label{eq:moment-scheme-strong}
\underset{n=0,\ldots,N}\sup~\Bigl(\E[|q_n^{\epsilon,\Delta t}|^2]+\E[|p_n^{\epsilon,\Delta t}|^2]\Bigr)\le C(T)\bigl(1+|q_0^\epsilon|^2+|p_0^\epsilon|^2\bigr).
\end{equation}
\end{propo}

\begin{proof}[Proof of Proposition~\ref{propo:moment-scheme-strong}]
Let $\bigl(Q_n^{\epsilon,\Delta t},P_n^{\epsilon,\Delta t}\bigr)_{n=0,\ldots,N}$ be given by~\eqref{eq:scheme-strong-QP}. One has the identities
\begin{equation}\label{eq:mildQP-strong}
\left\lbrace
\begin{aligned}
Q_n^{\epsilon,\Delta t}&=Q_0^{\epsilon,\Delta t}+\Delta t\sum_{k=0}^{n-1}f(Q_k^{\epsilon,\Delta t}-P_k^{\epsilon,\Delta t})+\sum_{k=0}^{n-1}\sigma(Q_k^{\epsilon,\Delta t}-P_k^{\epsilon,\Delta t})\Delta\beta_k\\
P_n^{\epsilon,\Delta t}&=\frac{1}{(1+\frac{\Delta t}{\epsilon^2})^n}P_0^{\epsilon,\Delta t}+\Delta t\sum_{k=0}^{n-1}\frac{1}{(1+\frac{\Delta t}{\epsilon^2})^{n-k}}f(Q_k^{\epsilon,\Delta t}-P_k^{\epsilon,\Delta t})+\sum_{k=0}^{n-1}\frac{1}{(1+\frac{\Delta t}{\epsilon^2})^{n-k}}\sigma(Q_k^{\epsilon,\Delta t}-P_k^{\epsilon,\Delta t})\Delta\beta_k
\end{aligned}
\right.
\end{equation}
for all $n\in\{0,\ldots,N\}$. Using It\^o's isometry formula and the Lipschitz continuity property of the mappings $f$ and $\sigma$ (Assumption~\ref{ass:f-sigma}), one obtains the inequality
\[
\E[|Q_n^{\epsilon,\Delta t}|^2]+\E[|P_n^{\epsilon,\Delta t}|^2]\le C(T)\Bigl(1+\E[|Q_0^{\epsilon,\Delta t}|^2]+\E[|P_0^{\epsilon,\Delta t}|^2]+\Delta t\sum_{k=0}^{n-1}\bigl(\E[|Q_k^{\epsilon,\Delta t}|^2]+\E[|P_k^{\epsilon,\Delta t}|^2]\bigr)\Bigr)
\]
for all $n\in\{0,\ldots,N\}$. Since $Q_0^{\epsilon,\Delta t}=q_0^\epsilon+\epsilon p_0^\epsilon$ and $P_0^{\epsilon,\Delta t}=\epsilon p_0^\epsilon$, applying discrete Gronwall's lemma yields the moment bound
\[
\underset{n=0,\ldots,N}\sup~\Bigl(\E[|Q_n^{\epsilon,\Delta t}|^2]+\E[|P_n^{\epsilon,\Delta t}|^2]\Bigr)\le C(T)\bigl(1+|q_0^\epsilon|^2+|p_0^\epsilon|^2\bigr).
\]
Using the identity $q_n^{\epsilon,\Delta t}=Q_n^{\epsilon,\Delta t}-P_n^{\epsilon,\Delta t}$ then gives
\[
\underset{n=0,\ldots,N}\sup~\E[|q_n^{\epsilon,\Delta t}|^2]\le C(T)\bigl(1+|q_0^\epsilon|^2+|p_0^\epsilon|^2\bigr).
\]
Using the identity $p_n^{\epsilon,\Delta t}=P_n^{\epsilon,\Delta t}/\epsilon$, one has
\[
p_n^{\epsilon,\Delta t}=\frac{1}{(1+\frac{\Delta t}{\epsilon^2})^n}p_0^{\epsilon}+\frac{\Delta t}{\epsilon}\sum_{k=0}^{n-1}\frac{1}{(1+\frac{\Delta t}{\epsilon^2})^{n-k}}f(q_k^{\epsilon,\Delta t})+\frac{1}{\epsilon}\sum_{k=0}^{n-1}\frac{1}{(1+\frac{\Delta t}{\epsilon^2})^{n-k}}\sigma(q_k^{\epsilon,\Delta t})\Delta\beta_k.
\]
Applying It\^o's isometry formula, and using the Lipschitz continuity properties of $f$ and $\sigma$ and the moment bounds for $q_n^{\epsilon,\Delta t}$ obtained above, one has, for all $n\in\{0,\ldots,N\}$
\begin{align*}
\E[|p_n^{\epsilon,\Delta t}|^2]&\le C|p_0^\epsilon|^2+\frac{CT\Delta t}{\epsilon^2}\sum_{k=0}^{n-1}\frac{1}{(1+\frac{\Delta t}{\epsilon^2})^{2(n-k)}}\bigl(1+\E[|q_k^{\epsilon,\Delta t}|^2]\bigr)\\
&\le C(T)\bigl(1+|q_0^\epsilon|^2+|p_0^\epsilon|^2\bigr)\bigl(1+\frac{\Delta t}{\epsilon^2}\sum_{k=0}^{n-1}\frac{1}{(1+\frac{\Delta t}{\epsilon^2})^{2(n-k)}}\bigr)\\
&\le C(T)\bigl(1+|q_0^\epsilon|^2+|p_0^\epsilon|^2\bigr),
\end{align*}
which gives the required moment bounds for $p_n^\epsilon$, using the inequality
\[
\tau\sum_{\ell=1}^{\infty}\frac{1}{(1+\tau)^{2\ell}}=\frac{1}{2+\tau}\le \frac12
\]
with $\tau=\Delta t/\epsilon^2\ge 0$. This concludes the proof of Proposition~\ref{propo:moment-scheme-strong}.
\end{proof}

\subsection{Moment bounds for the second numerical scheme}

\begin{propo}\label{propo:moment-scheme-weak}
Let $\bigl(q_n^{\epsilon,\Delta t},p_n^{\epsilon,\Delta t}\bigr)_{n=0,\ldots,N}$ be given by the numerical scheme~\eqref{eq:scheme-weak}. Let Assumption~\ref{ass:f-sigma} be satisfied. For all $T\in(0,\infty)$, there exists $C(T)\in(0,\infty)$ such that for all $\epsilon\in(0,\epsilon_0)$ and $\Delta t\in(0,\Delta t_0)$, one has
\begin{equation}\label{eq:moment-scheme-weak}
\underset{n=0,\ldots,N}\sup~\Bigl(\E[|q_n^{\epsilon,\Delta t}|^2]+\E[|p_n^{\epsilon,\Delta t}|^2]\Bigr)\le C(T)\bigl(1+|q_0^\epsilon|^2+|p_0^\epsilon|^2\bigr).
\end{equation}
\end{propo}

\begin{proof}[Proof of Proposition~\ref{propo:moment-scheme-strong}]
Let $\bigl(Q_n^{\epsilon,\Delta t},P_n^{\epsilon,\Delta t}\bigr)_{n=0,\ldots,N}$ be given by~\eqref{eq:scheme-weak-QP}. One has the identities
\begin{align*}
Q_n^{\epsilon,\Delta t}&=Q_0^{\epsilon,\Delta t}+\Delta t\sum_{k=0}^{n-1}f(Q_k^{\epsilon,\Delta t}-P_k^{\epsilon,\Delta t})+\sum_{k=0}^{n-1}\sigma(Q_k^{\epsilon,\Delta t}-P_k^{\epsilon,\Delta t})\Delta\beta_k\\
P_n^{\epsilon,\Delta t}&=e^{-\frac{n\Delta t}{\epsilon^2}}P_0^{\epsilon,\Delta t}+\Delta t\sum_{k=0}^{n-1}e^{-\frac{(n-k)\Delta t}{\epsilon^2}}f(Q_k^{\epsilon,\Delta t}-P_k^{\epsilon,\Delta t})+\sum_{k=0}^{n-1}e^{-\frac{(n-k)\Delta t}{\epsilon^2}}\sigma(Q_k^{\epsilon,\Delta t}-P_k^{\epsilon,\Delta t})\Delta\beta_k
\end{align*}
for all $n\in\{0,\ldots,N\}$. Using It\^o's isometry formula and the Lipschitz continuity property of the mappings $f$ and $\sigma$ (Assumption~\ref{ass:f-sigma}), one obtains the inequality
\[
\E[|Q_n^{\epsilon,\Delta t}|^2]+\E[|P_n^{\epsilon,\Delta t}|^2]\le C(T)\Bigl(1+\E[|Q_0^{\epsilon,\Delta t}|^2]+\E[|P_0^{\epsilon,\Delta t}|^2]+\Delta t\sum_{k=0}^{n-1}\bigl(\E[|Q_k^{\epsilon,\Delta t}|^2]+\E[|P_k^{\epsilon,\Delta t}|^2]\bigr)\Bigr)
\]
for all $n\in\{0,\ldots,N\}$. Since $Q_0^{\epsilon,\Delta t}=q_0^\epsilon+\epsilon p_0^\epsilon$ and $P_0^{\epsilon,\Delta t}=\epsilon p_0^\epsilon$, applying discrete Gronwall's lemma yields the moment bound
\[
\underset{n=0,\ldots,N}\sup~\Bigl(\E[|Q_n^{\epsilon,\Delta t}|^2]+\E[|P_n^{\epsilon,\Delta t}|^2]\Bigr)\le C(T)\bigl(1+|q_0^\epsilon|^2+|p_0^\epsilon|^2\bigr).
\]
Using the identity $q_n^{\epsilon,\Delta t}=Q_n^{\epsilon,\Delta t}-P_n^{\epsilon,\Delta t}$ then gives
\[
\underset{n=0,\ldots,N}\sup~\E[|q_n^{\epsilon,\Delta t}|^2]\le C(T)\bigl(1+|q_0^\epsilon|^2+|p_0^\epsilon|^2\bigr).
\]
Using the identity $p_n^{\epsilon,\Delta t}=P_n^{\epsilon,\Delta t}/\epsilon$, one has
\[
p_n^{\epsilon,\Delta t}=e^{-\frac{n\Delta t}{\epsilon^2}}p_0^{\epsilon}+\frac{\Delta t}{\epsilon}\sum_{k=0}^{n-1}e^{-\frac{(n-k)\Delta t}{\epsilon^2}}f(q_k^{\epsilon,\Delta t})+\frac{1}{\epsilon}\sum_{k=0}^{n-1}e^{-\frac{(n-k)\Delta t}{\epsilon^2}}\sigma(q_k^{\epsilon,\Delta t})\Delta\beta_k.
\]
Applying It\^o's isometry formula, and using the Lipschitz continuity properties of $f$ and $\sigma$ and the moment bounds for $q_n^{\epsilon,\Delta t}$ obtained above, one has, for all $n\in\{0,\ldots,N\}$
\begin{align*}
\E[|p_n^{\epsilon,\Delta t}|^2]&\le C|p_0^\epsilon|^2+\frac{CT\Delta t}{\epsilon^2}\sum_{k=0}^{n-1}e^{-\frac{2(n-k)\Delta t}{\epsilon^2}}\bigl(1+\E[|q_k^{\epsilon,\Delta t}|^2]\bigr)\\
&\le C(T)\bigl(1+|q_0^\epsilon|^2+|p_0^\epsilon|^2\bigr)\bigl(1+\frac{\Delta t}{\epsilon^2}\sum_{k=0}^{n-1}e^{-\frac{2(n-k)\Delta t}{\epsilon^2}}\bigr)\\
&\le C(T)\bigl(1+|q_0^\epsilon|^2+|p_0^\epsilon|^2\bigr),
\end{align*}
which gives the required moment bounds for $p_n^\epsilon$, using the inequality
\[
\tau\sum_{\ell=1}^{\infty}e^{-2\tau\ell}=\frac{\tau}{e^{2\tau}-1}\le C<\infty
\]
with $\tau=\Delta t/\epsilon^2\ge 0$. This concludes the proof of Proposition~\ref{propo:moment-scheme-strong}.
\end{proof}

\subsection{Proof of Proposition~\ref{propo:cv}}\label{sec:moments-cv}

As a consequence of Proposition~\ref{propo:moment-SDE}, we are now in position to provide the proof of the inequalities~\eqref{eq:cvSDE} and~\eqref{eq:cvscheme} from Proposition~\ref{propo:cv}.
\begin{proof}[Proof of the inequality~\eqref{eq:cvSDE}]
For all $t\in[0,T]$, the error is decomposed as
\begin{align*}
q^\epsilon(t)-q^0(t)&=Q^\epsilon(t)-q^0(t)-P^\epsilon(t)\\
&=q_0^\epsilon-q_0^0+\int_0^t \bigl(f(q^\epsilon(s))-f(q^0(s))\bigr)ds+\int_0^t \bigl(\sigma(q^\epsilon(s))-\sigma(q^0(s))\bigr)d\beta(s)-\epsilon p^\epsilon(t).
\end{align*}
Applying It\^o's isometry formula, using the Lipschitz continuity property of $f$ and $\sigma$ and the moment bound~\eqref{eq:moment-SDE}, one obtains, for all $t\in[0,T]$,
\[
\E[|q^\epsilon(t)-q^0(t)|^2]\le C(T)\Bigl(|q_0^\epsilon-q_0^0|^2+\int_0^t\E[|q^\epsilon(s)-q^0(s)|^2]ds+\epsilon^2\bigl(1+|q_0^\epsilon|^2+|p_0^\epsilon|^2\bigr)\Bigr).
\]
Applying Gronwall's lemma and using Assumption~\ref{ass:init}, the proof of the inequality~\eqref{eq:cvSDE} is completed.
\end{proof}

The proof of the inequality~\eqref{eq:cvscheme} is identical for the two choices of numerical schemes~\eqref{eq:scheme-strong} and~\eqref{eq:scheme-weak}, and is a variant of the proof of the inequality~\eqref{eq:cvSDE} in a discrete time situation.
\begin{proof}[Proof of the inequality~\eqref{eq:cvscheme}]
For all $n\in\{0,\ldots,N-1\}$, the error is decomposed as
\begin{align*}
q_n^{\epsilon,\Delta t}-q_n^{0,\Delta t}&=Q_n^{\epsilon,\Delta t}-q_n^{0,\Delta t}-P_n^{\epsilon,\Delta t}\\
&=q_0^\epsilon-q_0^0+\Delta t\sum_{k=0}^{n-1}\bigl(f(q_k^{\epsilon,\Delta t})-f(q_k^{0,\Delta t})+\sum_{k=0}^{n-1}\bigl(\sigma(q_k^{\epsilon,\Delta t})-\sigma(q_k^{0,\Delta t})\bigr)-\epsilon p_n^{\epsilon,\Delta t}.
\end{align*}
Applying It\^o's isometry formula, using the Lipschitz continuity property of $f$ and $\sigma$ and the moment bound~\eqref{eq:moment-scheme-strong}, one obtains, for all $n\in\{0,\ldots,N-1\}$,
\[
\E[|q_n^{\epsilon,\Delta t}-q_n^{0,\Delta t}|^2]\le C(T)\Bigl(|q_0^\epsilon-q_0^0|^2+\Delta t\sum_{k=0}^{n-1}\E[|q_k^{\epsilon,\Delta t}-q_k^{0,\Delta t}|^2]+\epsilon^2\bigl(1+|q_0^\epsilon|^2+|p_0^\epsilon|^2\bigr)\Bigr).
\]
Applying Gronwall's lemma and using Assumption~\ref{ass:init}, the proof of the inequality~\eqref{eq:cvscheme} is completed.
\end{proof}

\section{Proof of the strong error estimates}\label{sec:proofstrong}

Before proceeding with the proof of Theorem~\ref{theo:strong}, let us state two useful inequalities:
\begin{equation}\label{ineq1}
\underset{\tau\in[0,\infty)}\sup~\frac{|1-e^{-\tau}|}{\tau^{\frac12}}<\infty,
\end{equation}
and
\begin{equation}\label{ineq2}
\underset{n\in\N}\sup~\underset{\tau\in(0,\infty)}\sup~ (n+1)\bigl(\frac{1}{(1+\tau)^n}-e^{-n\tau}\bigr)<\infty.
\end{equation}
The proof of the inequality~\eqref{ineq1} is straightforward: the mapping $\tau\in[0,\infty)\to e^{-\tau}$ is bounded and Lipschitz continuous. To prove the inequality~\eqref{ineq2}, it suffices to check that the maximum of the function
\[
\tau\in[0,\infty)\mapsto \frac{1}{(1+\tau)^n}-e^{-n\tau}\in[0,\infty)
\]
is attained for a real number $\tau=\tau_n$ satisfying $e^{-n\tau_n}=\frac{1}{(1+\tau_n)^{n+1}}$: as a consequence
\[
\underset{\tau\in(0,\infty)}\sup~ \bigl(\frac{1}{(1+\tau)^n}-e^{-n\tau}\bigr)=\frac{1}{(1+\tau_n)^n}-e^{-n\tau_n}=\frac{\tau_n}{(1+\tau_n)^{n+1}}\le \frac{\tau_n}{(n+1)\tau_n}\le \frac{1}{n+1}.
\]

Before proceeding with the proof, let us also state and prove the following auxiliary result, concerning temporal regularity of the processes $Q^\epsilon$ and $P^\epsilon$, uniformly with respect to $\epsilon$.
\begin{lemma}\label{lem:reg}
Let Assumption~\ref{ass:f-sigma} be satisfied. For all $T\in(0,\infty)$, there exists $C(T)\in(0,\infty)$ such that for all $\epsilon\in(0,\epsilon_0)$ and all $s_1,s_2\in[0,T]$, one has
\begin{equation}\label{eq:lem-reg}
\E[|Q^\epsilon(s_2)-Q^\epsilon(s_1)|^2]+\E[|P^\epsilon(s_2)-P^\epsilon(s_1)|^2]\le C(T)(1+|q_0^\epsilon|^2+|p_0^\epsilon|^2)|s_2-s_1|.
\end{equation}
\end{lemma}

\begin{proof}[Proof of Lemma~\ref{lem:reg}]
On the one hand, for all $s_1,s_2\in[0,T]$, with $s_1\le s_2$, one has
\begin{align*}
Q^\epsilon(s_2)-Q^\epsilon(s_1)=\int_{s_1}^{s_2}f(q^\epsilon(s))ds+\int_{s_1}^{s_2}\sigma(q^\epsilon(s))d\beta(s),
\end{align*}
therefore using It\^o's isometry formula, the Lipschitz continuity property of the mappings $f$ and $\sigma$ (Assumption~\ref{ass:f-sigma}) and the moment bounds~\eqref{eq:moment-SDE} from Proposition~\ref{propo:moment-SDE}, one obtains
\[
\E[|Q^\epsilon(s_2)-Q^\epsilon(s_1)|^2]\le C(T)\int_{s_1}^{s_2}\bigl(1+\E[|q^\epsilon(s)|^2]\bigr)ds\le C(T)(1+|q_0^\epsilon|^2+|p_0^\epsilon|^2)|s_2-s_1|.
\]
On the one hand, for all $s_1,s_2\in[0,T]$, with $s_1\le s_2$, one has
\begin{align*}
P^\epsilon(s_2)-P^\epsilon(s_1)=\bigl(e^{-\frac{s_2-s_1}{\epsilon^2}}-1\bigr)P^\epsilon(s_1)+\int_{s_1}^{s_2}e^{-\frac{s_2-s}{\epsilon^2}}f(q^\epsilon(s))ds+\int_{s_1}^{s_2}e^{-\frac{s_2-s}{\epsilon^2}}\sigma(q^\epsilon(s))d\beta(s).
\end{align*}
Using the inequality~\eqref{ineq1}, the identity $P^\epsilon(s_1)=\epsilon p^\epsilon(s_1)$, It\^o's isometry formula, the Lipschitz continuity property of the mappings $f$ and $\sigma$ (Assumption~\ref{ass:f-sigma}) and the moment bounds~\eqref{eq:moment-SDE} from Proposition~\ref{propo:moment-SDE}, one obtains
\begin{align*}
\E[|P^\epsilon(s_2)-P^\epsilon(s_1)|^2]&\le C|s_2-s_1|\E[|p^\epsilon(s_1)|^2]+C(T)\int_{s_1}^{s_2}\bigl(1+\E[|q^\epsilon(s)|^2]\bigr)ds\\
&\le C(T)(1+|q_0^\epsilon|^2+|p_0^\epsilon|^2)|s_2-s_1|.
\end{align*}
Gathering the estimates then concludes the proof of Lemma~\ref{lem:reg}.
\end{proof}

We are now in position to provide the proof of Theorem~\ref{theo:strong}
\begin{proof}[Proof of Theorem~\ref{theo:strong}]
Let us first describe the decomposition of the error. Using the identities
\[
q^\epsilon(t_n)=Q^\epsilon(t_n)-P^\epsilon(t_n)~,\quad q_n^{\epsilon,\Delta t}=Q_n^{\epsilon,\Delta t}-P_n^{\epsilon,\Delta t},
\]
the mean-square error is bounded as follows:
\[
\E[|q_n^{\epsilon,\Delta t}-q^\epsilon(t_n)|^2]\le 2\E[|Q_n^{\epsilon,\Delta t}-Q^\epsilon(t_n)|^2]+2\E[|P_n^{\epsilon,\Delta t}-P^\epsilon(t_n)|^2].
\]
Let $E_n^{\epsilon,\Delta t}=\E[|Q_n^{\epsilon,\Delta t}-Q^\epsilon(t_n)|^2]+\E[|P_n^{\epsilon,\Delta t}-P^\epsilon(t_n)|^2]$. Recall the expressions~\eqref{eq:mildQP} for $Q^\epsilon(t_n)$ and $P^\epsilon(t_n)$, and the expressions~\eqref{eq:mildQP-strong} for $q_n^\epsilon$ and $p_n^\epsilon$.

On the one hand, one has
\begin{align*}
Q_n^{\epsilon,\Delta t}-Q^\epsilon(t_n)&=Q_0^{\epsilon,\Delta t}-Q^{\epsilon}(0)\\
&+\Delta t\sum_{k=0}^{n-1}f(Q_k^{\epsilon,\Delta t}-P_k^{\epsilon,\Delta t})-\int_0^{t_n}f(Q^\epsilon(t)-P^\epsilon(t))dt\\
&+\sum_{k=0}^{n-1}\sigma(Q_k^{\epsilon,\Delta t}-P_k^{\epsilon,\Delta t})\Delta\beta_k-\int_0^{t_n}\sigma(Q^\epsilon(t)-P^\epsilon(t))d\beta(t)\\
&=e_{n,Q,1,1}^{\epsilon,\Delta t}+e_{n,Q,1,2}^{\epsilon,\Delta t}+e_{n,Q,2,1}^{\epsilon,\Delta t}+e_{n,Q,2,2}^{\epsilon,\Delta t},
\end{align*}
where one sets
\begin{align*}
e_{n,Q,1,1}^{\epsilon,\Delta t}&=\sum_{k=0}^{n-1}\int_{t_k}^{t_{k+1}}\bigl(f(Q^{\epsilon}(t_k)-P^{\epsilon}(t_k))-f(Q^{\epsilon}(t)-P^{\epsilon}(t))\bigr)dt\\
e_{n,Q,1,2}^{\epsilon,\Delta t}&=\Delta t\sum_{k=0}^{n-1}\bigl(f(Q_k^{\epsilon,\Delta t}-P_k^{\epsilon,\Delta t})-f(Q^{\epsilon}(t_k)-P^{\epsilon}(t_k))\bigr)\\
e_{n,Q,2,1}^{\epsilon,\Delta t}&=\sum_{k=0}^{n-1}\int_{t_k}^{t_{k+1}}\bigl(\sigma(Q^{\epsilon}(t_k)-P^{\epsilon}(t_k))-\sigma(Q^{\epsilon}(t)-P^{\epsilon}(t))\bigr)d\beta(t)\\
e_{n,Q,2,2}^{\epsilon,\Delta t}&=\sum_{k=0}^{n-1}\bigl(\sigma(Q_k^{\epsilon,\Delta t}-P_k^{\epsilon,\Delta t})-\sigma(Q^{\epsilon}(t_k)-P^{\epsilon}(t_k))\bigr)\Delta\beta_k.
\end{align*}
On the other hand, one has
\begin{align*}
P_n^{\epsilon,\Delta t}-P^\epsilon(t_n)&=\frac{1}{(1+\frac{\Delta t}{\epsilon^2})^n}P_0^{\epsilon,\Delta t}-e^{-\frac{t_n}{\epsilon^2}}P_0^\epsilon\\
&+\Delta t\sum_{k=0}^{n-1}\frac{1}{(1+\frac{\Delta t}{\epsilon^2})^{n-k}}f(Q_k^{\epsilon,\Delta t}-P_k^{\epsilon,\Delta t})-\int_{0}^{t_n}e^{-\frac{t_n-t}{\epsilon^2}}f(Q^\epsilon(t)-P^\epsilon(t))dt\\
&+\sum_{k=0}^{n-1}\frac{1}{(1+\frac{\Delta t}{\epsilon^2})^{n-k}}\sigma(Q_k^{\epsilon,\Delta t}-P_k^{\epsilon,\Delta t})\Delta\beta_k-\int_{0}^{t_n}e^{-\frac{t_n-t}{\epsilon^2}}\sigma(Q^\epsilon(t)-P^\epsilon(t))d\beta(t)\\
&=e_{n,P,0}^{\epsilon,\Delta t}\\
&+e_{n,P,1,1}^{\epsilon,\Delta t}+e_{n,P,1,2}^{\epsilon,\Delta t}+e_{n,P,1,3}^{\epsilon,\Delta t}+e_{n,P,1,4}^{\epsilon,\Delta t}\\
&+e_{n,P,2,1}^{\epsilon,\Delta t}+e_{n,P,2,2}^{\epsilon,\Delta t}+e_{n,P,2,3}^{\epsilon,\Delta t}+e_{n,P,2,4}^{\epsilon,\Delta t}
\end{align*}
where one sets
\[
e_{n,P,0}^{\epsilon,\Delta t}=\bigl(\frac{1}{(1+\frac{\Delta t}{\epsilon^2})^n}-e^{-\frac{t_n}{\epsilon^2}}\bigr)P_0^{\epsilon,\Delta t}=\epsilon\bigl(\frac{1}{(1+\frac{\Delta t}{\epsilon^2})^n}-e^{-\frac{t_n}{\epsilon^2}}\bigr)p_0^{\epsilon},
\]
then
\begin{align*}
e_{n,P,1,1}^{\epsilon,\Delta t}&=\sum_{k=0}^{n-1}\int_{t_k}^{t_{k+1}}e^{-\frac{t_n-t}{\epsilon^2}}\bigl(f(Q^\epsilon(t_k)-P^\epsilon(t_k))-f(Q^\epsilon(t)-P^\epsilon(t))\bigr)dt\\
e_{n,P,1,2}^{\epsilon,\Delta t}&=\sum_{k=0}^{n-1}\int_{t_k}^{t_{k+1}}e^{-\frac{t_n-t}{\epsilon^2}}\bigl(f(Q_k^{\epsilon,\Delta t}-P_k^{\epsilon,\Delta t})-f(Q^\epsilon(t_k)-P^\epsilon(t_k))\bigr)dt\\
e_{n,P,1,3}^{\epsilon,\Delta t}&=\sum_{k=0}^{n-1}\int_{t_k}^{t_{k+1}}\bigl(e^{-\frac{t_n-t_k}{\epsilon^2}}-e^{-\frac{t_n-t}{\epsilon^2}}\bigr)f(Q_k^{\epsilon,\Delta t}-P_k^{\epsilon,\Delta t})dt\\
e_{n,P,1,4}^{\epsilon,\Delta t}&=\Delta t\sum_{k=0}^{n-1}\bigl(\frac{1}{(1+\frac{\Delta t}{\epsilon^2})^{n-k}}-e^{-\frac{t_n-t_k}{\epsilon^2}}\bigr)f(Q_k^{\epsilon,\Delta t}-P_k^{\epsilon,\Delta t})
\end{align*}
and finally
\begin{align*}
e_{n,P,2,1}^{\epsilon,\Delta t}&=\sum_{k=0}^{n-1}\int_{t_k}^{t_{k+1}}e^{-\frac{t_n-t}{\epsilon^2}}\bigl(\sigma(Q^\epsilon(t_k)-P^\epsilon(t_k))-\sigma(Q^\epsilon(t)-P^\epsilon(t))\bigr)d\beta(t)\\
e_{n,P,2,2}^{\epsilon,\Delta t}&=\sum_{k=0}^{n-1}\int_{t_k}^{t_{k+1}}e^{-\frac{t_n-t}{\epsilon^2}}\bigl(\sigma(Q_k^{\epsilon,\Delta t}-P_k^{\epsilon,\Delta t})-\sigma(Q^\epsilon(t_k)-P^\epsilon(t_k))\bigr)d\beta(t)\\
e_{n,P,2,3}^{\epsilon,\Delta t}&=\sum_{k=0}^{n-1}\int_{t_k}^{t_{k+1}}\bigl(e^{-\frac{t_n-t_k}{\epsilon^2}}-e^{-\frac{t_n-t}{\epsilon^2}}\bigr)\sigma(Q_k^{\epsilon,\Delta t}-P_k^{\epsilon,\Delta t})d\beta(t)\\
e_{n,P,2,4}^{\epsilon,\Delta t}&=\sum_{k=0}^{n-1}\bigl(\frac{1}{(1+\frac{\Delta t}{\epsilon^2})^{n-k}}-e^{-\frac{t_n-t_k}{\epsilon^2}}\bigr)\sigma(Q_k^{\epsilon,\Delta t}-P_k^{\epsilon,\Delta t})\Delta\beta_k.
\end{align*}

Let us prove error estimates for each of the terms defined above. Let us start with the terms appearing in the right-hand side of the expression of the error term $Q_n^{\epsilon,\Delta t}-Q^\epsilon(t_n)$ above.

Using the Cauchy--Schwarz inequality, the Lipschitz continuity of $f$ (Assumption~\ref{ass:f-sigma}) and the inequality~\eqref{eq:lem-reg} from Lemma~\ref{lem:reg}, one obtains
\begin{align*}
\E[|e_{n,Q,1,1}^{\epsilon,\Delta t}|^2]&\le CT\sum_{k=0}^{n-1}\int_{t_k}^{t_{k+1}}\bigl(\E[|Q^\epsilon(t_k)-Q^\epsilon(t)|^2]+\E[|P^\epsilon(t_k)-P^\epsilon(t)|^2]\bigr)dt\\
&\le C(T)(1+|q_0^\epsilon|^2+|p_0^\epsilon|^2)\Delta t.
\end{align*}

Using the Cauchy--Schwarz inequality and the Lipschitz continuity of $f$, one obtains
\[
\E[|e_{n,Q,1,2}^{\epsilon,\Delta t}|^2]\le CT\Delta t\sum_{k=0}^{n-1}E_k^{\epsilon,\Delta t}.
\]

Using It\^o's isometry formula, the Lipschitz continuity of $\sigma$ and Lemma~\eqref{eq:lem-reg}, one obtains
\begin{align*}
\E[|e_{n,Q,2,1}^{\epsilon,\Delta t}|^2]&\le C\sum_{k=0}^{n-1}\int_{t_k}^{t_{k+1}}\bigl(\E[|Q^\epsilon(t_k)-Q^\epsilon(t)|^2]+\E[|P^\epsilon(t_k)-P^\epsilon(t)|^2]\bigr)dt\\
&\le C(T)(1+|q_0^\epsilon|^2+|p_0^\epsilon|^2)\Delta t.
\end{align*}
Using It\^o's isometry formula and the Lipschitz continuity of $\sigma$, one obtains
\[
\E[|e_{n,Q,2,2}^{\epsilon,\Delta t}|^2]\le CT\Delta t\sum_{k=0}^{n-1}E_k^{\epsilon,\Delta t}.
\]
Gathering the estimates, for all $n\in\{1,\ldots,N\}$ one obtains the upper bound
\begin{equation}\label{eq:aux-strongQ}
\begin{aligned}
\E[|Q_n^{\epsilon,\Delta t}-Q^\epsilon(t_n)|^2]&\le 4\Bigl(\E[|e_{n,Q,1,1}^{\epsilon,\Delta t}|^2]+\E[|e_{n,Q,1,2}^{\epsilon,\Delta t}|^2]+\E[|e_{n,Q,2,1}^{\epsilon,\Delta t}|^2]+\E[|e_{n,Q,2,2}^{\epsilon,\Delta t}|^2]\Bigr)\\
&\le CT\Delta t\sum_{k=0}^{n-1}E_k^{\epsilon,\Delta t}+C(T)(1+|q_0^\epsilon|^2+|p_0^\epsilon|^2)\Delta t.
\end{aligned}
\end{equation}

Let us now treat the terms appearing in the right-hand side of the expression of the error term $P_n^{\epsilon,\Delta t}-P^\epsilon(t_n)$ above.

Using the inequality~\eqref{ineq2}, one obtains
\[
\E[|e_{n,P,0}^{\epsilon,\Delta t}|^2]\le \frac{C\epsilon^2}{(n+1)^2}|p_0^\epsilon|^2.
\]

Using the Cauchy--Schwarz inequality, the Lipschitz continuity of $f$ and Lemma~\eqref{eq:lem-reg}, one obtains
\begin{align*}
\E[|e_{n,P,1,1}^{\epsilon,\Delta t}|^2]&\le CT\sum_{k=0}^{n-1}\int_{t_k}^{t_{k+1}}\bigl(\E[|Q^\epsilon(t_k)-Q^\epsilon(t)|^2]+\E[|P^\epsilon(t_k)-P^\epsilon(t)|^2]\bigr)dt\\
&\le C(T)(1+|q_0^\epsilon|^2+|p_0^\epsilon|^2)\Delta t.
\end{align*}
Using the Cauchy--Schwarz inequality and the Lipschitz continuity of $f$, one obtains
\[
\E[|e_{n,P,1,2}^{\epsilon,\Delta t}|^2]\le CT\Delta t\sum_{k=0}^{n-1}E_k^{\epsilon,\Delta t}.
\]
Using the Cauchy--Schwarz inequality, the Lipschitz continuity of $f$ and the inequality~\eqref{ineq1}, one obtains
\begin{align*}
\E[|e_{n,P,1,3}^{\epsilon,\Delta t}|^2]&\le C(T)(1+|q_0^\epsilon|^2+|p_0^\epsilon|^2)\sum_{k=0}^{n-1}\int_{t_k}^{t_{k+1}}\bigl(e^{-\frac{t-t_k}{\epsilon^2}}-1 \bigr)^2e^{-2\frac{t_n-t}{\epsilon^2}}dt\\
&\le C(T)(1+|q_0^\epsilon|^2+|p_0^\epsilon|^2)\frac{\Delta t}{\epsilon^2}\int_{0}^{t_{n}}e^{-2\frac{t_n-t}{\epsilon^2}}dt\\
&\le C(T)(1+|q_0^\epsilon|^2+|p_0^\epsilon|^2)\Delta t.
\end{align*}
Using the Cauchy--Schwarz inequality, the Lipschitz continuity of $f$ and the inequality~\eqref{ineq2}, one obtains
\begin{align*}
\E[|e_{n,P,1,4}^{\epsilon,\Delta t}|^2]&\le C(T)(1+|q_0^\epsilon|^2+|p_0^\epsilon|^2)\sum_{k=0}^{n-1}\frac{\Delta t}{(n-k)^2}\\
&\le C(T)(1+|q_0^\epsilon|^2+|p_0^\epsilon|^2)\Delta t.
\end{align*}

Using It\^o's isometry formula, the Lipschitz continuity of $\sigma$ and Lemma~\eqref{eq:lem-reg}, one obtains
\begin{align*}
\E[|e_{n,P,2,1}^{\epsilon,\Delta t}|^2]&\le CT\sum_{k=0}^{n-1}\int_{t_k}^{t_{k+1}}\bigl(\E[|Q^\epsilon(t_k)-Q^\epsilon(t)|^2]+\E[|P^\epsilon(t_k)-P^\epsilon(t)|^2]\bigr)dt\\
&\le C(T)(1+|q_0^\epsilon|^2+|p_0^\epsilon|^2)\Delta t.
\end{align*}
Using It\^o's isometry formula and the Lipschitz continuity of $\sigma$, one obtains
\[
\E[|e_{n,P,2,2}^{\epsilon,\Delta t}|^2]\le CT\Delta t\sum_{k=0}^{n-1}E_k^{\epsilon,\Delta t}.
\]
Using It\^o's isometry formula, the Lipschitz continuity of $\sigma$ and the inequality~\eqref{ineq1}, one obtains
\begin{align*}
\E[|e_{n,P,2,3}^{\epsilon,\Delta t}|^2]&\le C(T)(1+|q_0^\epsilon|^2+|p_0^\epsilon|^2)\sum_{k=0}^{n-1}\int_{t_k}^{t_{k+1}}\bigl(e^{-\frac{t-t_k}{\epsilon^2}}-1 \bigr)^2e^{-2\frac{t_n-t}{\epsilon^2}}dt\\
&\le C(T)(1+|q_0^\epsilon|^2+|p_0^\epsilon|^2)\frac{\Delta t}{\epsilon^2}\int_{0}^{t_{n}}e^{-2\frac{t_n-t}{\epsilon^2}}dt\\
&\le C(T)(1+|q_0^\epsilon|^2+|p_0^\epsilon|^2)\Delta t.
\end{align*}
Using It\^o's isometry formula, the Lipschitz continuity of $\sigma$ and the inequality~\eqref{ineq2}, one obtains
\begin{align*}
\E[|e_{n,P,2,4}^{\epsilon,\Delta t}|^2]&\le C(T)(1+|q_0^\epsilon|^2+|p_0^\epsilon|^2)\sum_{k=0}^{n-1}\frac{\Delta t}{(n-k)^2}\\
&\le C(T)(1+|q_0^\epsilon|^2+|p_0^\epsilon|^2)\Delta t.
\end{align*}
Gathering the estimates, for all $n\in\{1,\ldots,N\}$ one obtains the upper bound
\begin{equation}\label{eq:aux-strongP}
\begin{aligned}
\E[|P_n^{\epsilon,\Delta t}-P^\epsilon(t_n)|^2]&\le 9\E[|e_{n,P,0}^{\epsilon,\Delta t}|^2]\\
&+9\Bigl(\E[|e_{n,P,1,1}^{\epsilon,\Delta t}|^2]+\E[|e_{n,P,1,2}^{\epsilon,\Delta t}|^2]+\E[|e_{n,P,1,3}^{\epsilon,\Delta t}|^2]+\E[|e_{n,P,1,4}^{\epsilon,\Delta t}|^2]\Bigr)\\
&+9\Bigl(\E[|e_{n,P,2,1}^{\epsilon,\Delta t}|^2]+\E[|e_{n,P,2,2}^{\epsilon,\Delta t}|^2]+\E[|e_{n,P,2,3}^{\epsilon,\Delta t}|^2]+\E[|e_{n,P,2,4}^{\epsilon,\Delta t}|^2]\Bigr)\\
&\le C\Delta t\sum_{k=1}^{n-1}E_k^{\epsilon,\Delta t}+\frac{C\epsilon^2}{(n+1)^2}|p_0^\epsilon|^2+C(T)(1+|q_0^\epsilon|^2+|p_0^\epsilon|^2)\Delta t.
\end{aligned}
\end{equation}

Let us now conclude the proof of the strong error estimates. Owing to the upper bounds~\eqref{eq:aux-strongQ} and~\eqref{eq:aux-strongP}, for all $n\in\{0,\ldots,N\}$, one obtains the upper bound
\begin{align*}
E_n^{\epsilon,\Delta t}&=\E[|Q_n^{\epsilon,\Delta t}-Q^\epsilon(t_n)|^2]+\E[|P_n^{\epsilon,\Delta t}-P^\epsilon(t_n)|^2]\\
&\le C\Delta t\sum_{k=1}^{n-1}E_k^{\epsilon,\Delta t}+\frac{C\epsilon^2}{(n+1)^2}|p_0^\epsilon|^2+C(T)(1+|q_0^\epsilon|^2+|p_0^\epsilon|^2)\Delta t,
\end{align*}
with $E_0^{\epsilon,\Delta t}=0$.

For all $n\in\{1,\ldots,N\}$, set
\[
\hat{E}_{n}^{\epsilon,\Delta t}=E_n^{\epsilon,\Delta t}-\frac{C\epsilon^2}{(n+1)^2}|p_0^\epsilon|^2,
\]
then one obtains the inequality
\[
\hat{E}_n^{\epsilon,\Delta t}\le C(T)\Delta t\sum_{k=1}^{n-1}\hat{E}_k^{\epsilon,\Delta t}+C(T)(1+|q_0^\epsilon|^2+|p_0^\epsilon|^2)\Delta t,
\]
using the fact that $\sum_{k=1}^{\infty}1/k^2<\infty$. Applying discrete Gronwall's lemma yields the inequality
\[
\underset{n=1,\ldots,N}\sup~\hat{E}_{n}^{\epsilon,\Delta t}\le C(T)(1+|q_0^\epsilon|^2+|p_0^\epsilon|^2)\Delta t,
\]
therefore one obtains for all $n\in\{1,\ldots,N\}$ the upper bound
\[
E_{n}^{\epsilon,\Delta t}\le \frac{C\epsilon^2}{(n+1)^2}|p_0^\epsilon|^2+C(T)(1+|q_0^\epsilon|^2+|p_0^\epsilon|^2)\Delta t.
\]
Using Assumption~\ref{ass:init} then concludes the proof of Theorem~\ref{theo:strong}.
\end{proof}

\section{Proof of the weak error estimates}\label{sec:proofweak}

The objective of this section is to provide the proof of Theorem~\ref{theo:weak}. The most important

\subsection{Regularity estimates for solutions of Kolmogorov equations}\label{sec:Kolmogorov}

Let $\varphi:\R^d\to\R$ be a mapping of class $\mathcal{C}^3$, with bounded derivatives of order $1,2,3$. For all $(q,p)\in\R^{2d}$ and $t\ge 0$, let
\begin{equation}\label{eq:u}
u^\epsilon(t,q,p)=\E_{q,p}[\varphi(q^\epsilon(t))]
\end{equation}
where the subscript means that the process $\bigl(q^\epsilon(t),p^\epsilon(t)\bigr)_{t\ge 0}$ is solution of the SDE~\eqref{eq:SDE} with initial values $q^\epsilon(0)=q$ and $p^\epsilon(0)=p$.

As a consequence of the definition~\eqref{eq:u} of the function $u$, the weak error may be written as
\begin{equation}\label{eq:weakerror}
\E[\varphi(q_n^{\epsilon,\Delta t})]-\E[\varphi(q^\epsilon(t_n))]
=\E[u^\epsilon(0,q_n^{\epsilon,\Delta t},p_n^{\epsilon,\Delta t})]
-\E[u^\epsilon(t_n,q_0^{\epsilon,\Delta t},p_0^{\epsilon,\Delta t})],
\end{equation}
for all $n\in\{0,\ldots,N\}$, where $q_n^{\epsilon,\Delta t}$ and $p_n^{\epsilon,\Delta t}$ are given by~\eqref{eq:scheme-weak}, see for instance~\cite{Talay:86,TalayTubaro}.

Since the mappings $\varphi$, $f$ and $\sigma$ are of class $\mathcal{C}^3$ with bounded derivatives, the mapping
\[
(t,q,p)\in [0,\infty)\times \R^{2d}\mapsto u^\epsilon(t,q,p)
\]
is of class $\mathcal{C}_t^1\mathcal{C}_{q,p}^3$, moreover it is the solution of the Kolmogorov equation (see~\cite{Cerrai} for instance)
\begin{equation}\label{eq:Kolmogorov}
\left\lbrace
\begin{aligned}
\partial_tu^\epsilon(t,q,p)&=\frac{1}{\epsilon}\Bigl(\nabla_qu^\epsilon(t,q,p)\cdot p+\nabla_pu^\epsilon(t,q,p)\cdot f(q)\Bigr)\\
&+\frac{1}{\epsilon^2}\Bigl(-\nabla_pu^\epsilon(t,q,p)\cdot p+\nabla_{p}^2u^\epsilon(t,q,p):a(q)\Bigr),\\
u^\epsilon(0,q,p)&=\varphi(q),
\end{aligned}
\right.
\end{equation}
where the mapping $a$ is given by~\eqref{eq:a} and the following notation is used:
\[
\nabla_{p}^2u^\epsilon(t,q,p):a(q)=\sum_{i,j=1}^{d}\partial_i\partial_j u^\epsilon(t,q,p)a_{ij}(q).
\]

Let us state upper bounds on the first, second and third order spatial derivatives of $u^\epsilon$, with a careful analysis of the dependence with respect to the parameter $\epsilon$.
\begin{propo}\label{propo:Kolmogorov}

For all $T\in(0,\infty)$, there exists $C(T)\in(0,\infty)$ such that for any function $\varphi:\R^d\to\R$ of class $\mathcal{C}^3$ with bounded derivatives of order $1,2,3$ and for all $h^1,h^2,h^3\in\R^d$, one has
\begin{equation}\label{eq:K1}
\underset{\epsilon\in(0,\epsilon_0)}\sup~\underset{(t,q,p)\in[0,T]\times\R^{2d}}\sup~\frac{|\nabla_pu^\epsilon(t,q,p)\cdot h^1|}{\epsilon}\le C(T)\vvvert\varphi\vvvert_1|h^1|.
\end{equation}

\begin{equation}\label{eq:K2}
\underset{\epsilon\in(0,\epsilon_0)}\sup~\underset{(t,q,p)\in[0,T]\times\R^{2d}}\sup~\Bigl(\frac{|\nabla_q\nabla_pu^\epsilon(t,q,p)\cdot (h^1,h^2)|}{\epsilon}+\frac{|\nabla_p^2u^\epsilon(t,q,p)\cdot (h^1,h^2)|}{\epsilon^2}\Bigr)\le C(T)\vvvert\varphi\vvvert_2|h^1||h^2|.
\end{equation}

\begin{equation}\label{eq:K3}
\underset{\epsilon\in(0,\epsilon_0)}\sup~\underset{(t,q,p)\in[0,T]\times\R^{2d}}\sup~\Bigl(\frac{|\nabla_q\nabla_p^2u^\epsilon(t,q,p)\cdot (h^1,h^2,h^3)|}{\epsilon^2}+\frac{|\nabla_p^3u^\epsilon(t,q,p)\cdot (h^1,h^2,h^3)|}{\epsilon^3}\Bigr)\le C(T)\vvvert\varphi\vvvert_3|h^1||h^2||h^3|.
\end{equation}

\end{propo}

The proof of Proposition~\ref{propo:Kolmogorov} is postponed to Section~\ref{sec:proofKolmogorov}.

\subsection{Proof of Theorem~\ref{theo:weak}}

This section is devoted to the proof of Theorem~\ref{theo:weak}. Many of the arguments are standard, the novelty is to deal carefully with the dependence with respect to $\epsilon$ using the regularity estimates from Proposition~\ref{propo:Kolmogorov} above.

Before proceeding with the proof of Theorem~\ref{theo:weak}, let us state and prove an auxiliary result, which is a variant of Lemma~\ref{lem:reg} for the auxiliary processes $\tilde{q}^{\epsilon,\Delta t}$ and $\tilde{p}^{\epsilon,\Delta t}$.
\begin{lemma}\label{lem:regtilde}
Let Assumption~\ref{ass:f-sigma} be satisfied. For all $T\in(0,\infty)$, there exists $C(T)\in(0,\infty)$ such that for all $\epsilon\in(0,\epsilon_0)$, all $\Delta t\in(0,\Delta t_0)$, all $n\in\{0,\ldots,N-1\}$ and all $t\in[t_n,t_{n+1}]$, one has
\begin{equation}\label{eq:lem-regtilde}
\E[|\tilde{q}^{\epsilon,\Delta t}(t)-\tilde{q}^{\epsilon,\Delta t}(t_n)|^2]+\epsilon^2\E[|\tilde{p}^{\epsilon,\Delta t}(t)-\tilde{p}^{\epsilon,\Delta t}(t_n)|^2]\le C(T)(1+|q_0^\epsilon|^2+|p_0^\epsilon|^2)\Delta t.
\end{equation}
\end{lemma}

\begin{proof}[Proof of Lemma~\ref{lem:regtilde}]
Introduce the auxiliary variables $\tilde{Q}^{\epsilon,\Delta t}(t)=\tilde{q}^{\epsilon,\Delta t}(t)+\epsilon\tilde{p}^{\epsilon,\Delta t}(t)$ and $\tilde{P}^{\epsilon,\Delta t}(t)=\epsilon\tilde{p}^{\epsilon}(t)$. The inequality~\eqref{eq:lem-regtilde} is a straightforward consequence of the following claim: for all $T\in(0,\infty)$, there exists $C(T)\in(0,\infty)$ such that for all $\epsilon\in(0,\epsilon_0)$, all $\Delta t\in(0,\Delta t_0)$ and all $n\in\{0,\ldots,N-1\}$ and all $t\in[t_n,t_{n+1}]$, one has
\begin{equation}\label{eq:lem-regtilde-claim}
\E[|\tilde{Q}^{\epsilon,\Delta t}(t)-\tilde{Q}^{\epsilon,\Delta t}(t_n)|^2]+\E[|\tilde{P}^{\epsilon,\Delta t}(t)-\tilde{P}^{\epsilon,\Delta t}(t_n)|^2]\le C(T)(1+|q_0^\epsilon|^2+|p_0^\epsilon|^2)\Delta t.
\end{equation}
Let us establish the claim~\eqref{eq:lem-regtilde-claim}. Recall that $\tilde{q}^{\epsilon,\Delta t}(t_n)=q_n^{\epsilon,\Delta t}$ and $\tilde{p}^{\epsilon,\Delta t}(t_n)=p_n^{\epsilon,\Delta t}$ for all $n\in\{0,\ldots,N-1\}$, owing to Proposition~\ref{propo:SDEtilde}.

On the one hand, for all $t\in[t_n,t_{n+1}]$, one has
\begin{align*}
\tilde{Q}^{\epsilon,\Delta t}(t)-\tilde{Q}^{\epsilon,\Delta t}(t_n)=\int_{s_1}^{s_2}f(q_n^{\epsilon,\Delta t})ds+\int_{s_1}^{s_2}\sigma(q_n^{\epsilon,\Delta t})d\beta(s),
\end{align*}
therefore using It\^o's isometry formula, the Lipschitz continuity property of the mappings $f$ and $\sigma$ (Assumption~\ref{ass:f-sigma}) and the moment bounds~\eqref{eq:moment-scheme-weak} from Proposition~\ref{propo:moment-scheme-weak}, one obtains
\[
\E[|\tilde{Q}^{\epsilon,\Delta t}(t)-\tilde{Q}^{\epsilon,\Delta t}(t_n)|^2]\le C(T)\int_{t_n}^{t}\bigl(1+\E[|q_n^{\epsilon,\Delta t}|^2]\bigr)ds\le C(T)(1+|q_0^\epsilon|^2+|p_0^\epsilon|^2)\Delta t.
\]
On the other hand, for all $t\in[t_n,t_{n+1}]$, one has
\begin{align*}
\tilde{P}^{\epsilon,\Delta t}(t)-\tilde{P}^{\epsilon,\Delta t}(t_n)=\bigl(e^{-\frac{t-t_n}{\epsilon^2}}-1\bigr)\tilde{P}^{\epsilon,\Delta t}(t_n)+\int_{t_n}^{t}e^{-\frac{t-s}{\epsilon^2}}f(q_n^{\epsilon,\Delta t})ds+\int_{t_n}^{t}e^{-\frac{t-s}{\epsilon^2}}\sigma(q_n^{\epsilon,\Delta t})d\beta(s).
\end{align*}
Using the inequality~\eqref{ineq1}, the identity $\tilde{P}^{\epsilon,\Delta t}(t_n)=\epsilon p_n^{\epsilon,\Delta t}$, It\^o's isometry formula, the Lipschitz continuity property of the mappings $f$ and $\sigma$ (Assumption~\ref{ass:f-sigma}) and the moment bounds~\eqref{eq:moment-scheme-weak} from Proposition~\ref{propo:moment-scheme-weak}, one obtains
\begin{align*}
\E[|\tilde{P}^{\epsilon,\Delta t}(t)-\tilde{P}^{\epsilon,\Delta t}(t_n)|^2]&\le C\Delta t\E[|p_n^{\epsilon,\Delta t}|^2]+C(T)\int_{t_n}^{t}\bigl(1+\E[|q_n^{\epsilon,\Delta t}|^2]\bigr)ds\\
&\le C(T)(1+|q_0^\epsilon|^2+|p_0^\epsilon|^2)\Delta t.
\end{align*}
Gathering the estimates yields the claim~\eqref{eq:lem-regtilde-claim} and then concludes the proof of Lemma~\ref{lem:reg}.
\end{proof}

\begin{proof}[Proof of Theorem~\ref{theo:weak}]
Owing to the expression~\eqref{eq:weakerror} of the weak error using the function $u^\epsilon$ defined by~\eqref{eq:u}, the auxiliary process defined by~\eqref{eq:SDEtilde} and a standard telescoping sum argument, one has the following decomposition of the weak error: for all $n\in\{0,\ldots,N\}$,
\begin{equation}\label{eq:weakerrordecomp}
\begin{aligned}
\E[\varphi(q_n^{\epsilon,\Delta t})]&-\E[\varphi(q^\epsilon(t_n))]
=\E[u^\epsilon(0,q_n^{\epsilon,\Delta t},p_n^{\epsilon,\Delta t})]
-\E[u^\epsilon(t_n,q_0^{\epsilon,\Delta t},p_0^{\epsilon,\Delta t})]\\
&=\sum_{m=0}^{n-1}\bigl(\E[u^\epsilon(t_n-t_{m+1},q_{m+1}^{\epsilon,\Delta t},p_{m+1}^{\epsilon,\Delta t})]-\E[u^\epsilon(t_n-t_{m},q_{m}^{\epsilon,\Delta t},p_{m}^{\epsilon,\Delta t})]\bigr)\\
&=\sum_{m=0}^{n-1}\bigl(\E[u^\epsilon(t_n-t_{m+1},\tilde{q}^{\epsilon,\Delta t}(t_{m+1}),\tilde{p}^{\epsilon,\Delta t}(t_{m+1}))]-\E[u^\epsilon(t_n-t_{m},\tilde{q}^{\epsilon,\Delta t}(t_{m}),\tilde{p}^{\epsilon,\Delta t}(t_{m}))]\bigr).
\end{aligned}
\end{equation}
Using It\^o's formula, the fact that $u^\epsilon$ is solution of the Kolmogorov equation~\eqref{eq:Kolmogorov} and the fact that the process $\bigl(\tilde{q}^{\epsilon,\Delta t}(t),\tilde{p}^{\epsilon,\Delta t}(t)\bigr)_{t\ge 0}$ is solution of the SDE~\eqref{eq:SDEtilde}, for all $m\in\{0,\ldots,n-1\}$, one obtains
\[
\E[u^\epsilon(t_n-t_{m+1},\tilde{q}^{\epsilon,\Delta t}(t_{m+1}),\tilde{p}^{\epsilon,\Delta t}(t_{m+1}))]-\E[u^\epsilon(t_n-t_{m},\tilde{q}^{\epsilon,\Delta t}(t_{m}),\tilde{p}^{\epsilon,\Delta t}(t_{m}))]=\mathcal{E}_{m,n,1}^{\epsilon,\Delta t}+\mathcal{E}_{m,n,2}^{\epsilon,\Delta t}
\]
where the error terms in the right-hand side above are defined for all $0\le m\le n-1$ by
\begin{align*}
\mathcal{E}_{m,n,1}^{\epsilon,\Delta t}&=\frac{1}{\epsilon}\int_{t_m}^{t_{m+1}}\E[\nabla_pu^\epsilon(t_n-t,\tilde{q}^{\epsilon,\Delta t}(t),\tilde{p}^{\epsilon,\Delta t}(t))\cdot\bigl(f(\tilde{q}^{\epsilon,\Delta t}(t_m))-f(\tilde{q}^{\epsilon,\Delta t}(t))\bigr)]dt\\
\mathcal{E}_{m,n,2}^{\epsilon,\Delta t}&=\frac{1}{\epsilon^2}\int_{t_m}^{t_{m+1}}\E[\nabla_{p}^2u^\epsilon(t_n-t,\tilde{q}^{\epsilon,\Delta t}(t),\tilde{p}^{\epsilon,\Delta t}(t)):\bigl(a(\tilde{q}^{\epsilon,\Delta t}(t_m))-a(\tilde{q}^{\epsilon,\Delta t}(t))\bigr)]dt.
\end{align*}
Observe that the error term $\mathcal{E}_{m,n,1}^{\epsilon,\Delta t}$, resp. $\mathcal{E}_{m,n,2}^{\epsilon,\Delta t}$, vanishes if the mapping $f$ is constant, resp. if the mapping $\sigma$ is constant. This is due to the construction of the numerical scheme~\eqref{eq:scheme-weak}.

The error terms $\mathcal{E}_{m,n,1}^{\epsilon,\Delta t}$ and $\mathcal{E}_{m,n,2}^{\epsilon,\Delta t}$ are then decomposed as follows: set
\begin{align*}
\mathcal{E}_{m,n,1,1}^{\epsilon,\Delta t}&=\frac{1}{\epsilon}\int_{t_m}^{t_{m+1}}\E[\nabla_pu^\epsilon(t_n-t,\tilde{q}^{\epsilon,\Delta t}(t_m),\tilde{p}^{\epsilon,\Delta t}(t_m))\cdot\bigl(f(\tilde{q}^{\epsilon,\Delta t}(t_m))-f(\tilde{q}^{\epsilon,\Delta t}(t))\bigr)]dt\\
\mathcal{E}_{m,n,1,2}^{\epsilon,\Delta t}&=\mathcal{E}_{m,n,1}^{\epsilon,\Delta t}-\mathcal{E}_{m,n,1,1}^{\epsilon,\Delta t},\\
\mathcal{E}_{m,n,2,1}^{\epsilon,\Delta t}&=\frac{1}{\epsilon^2}\int_{t_m}^{t_{m+1}}\E[\nabla_{p}^2u^\epsilon(t_n-t,\tilde{q}^{\epsilon,\Delta t}(t_m),\tilde{p}^{\epsilon,\Delta t}(t_m)):\bigl(a(\tilde{q}^{\epsilon,\Delta t}(t_m))-a(\tilde{q}^{\epsilon,\Delta t}(t))\bigr)]dt\\
\mathcal{E}_{m,n,2,2}^{\epsilon,\Delta t}&=\mathcal{E}_{m,n,2}^{\epsilon,\Delta t}-\mathcal{E}_{m,n,2,1}^{\epsilon,\Delta t}.
\end{align*}

It remains to obtain upper bounds for the four error terms.

$\bullet$ Since the mapping $f$ is of class $\mathcal{C}^2$ with bounded second-order derivative, a Taylor expansion argument gives
\[
\mathcal{E}_{m,n,1,1}^{\epsilon,\Delta t}=\mathcal{E}_{m,n,1,1,1}^{\epsilon,\Delta t}+\mathcal{E}_{m,n,1,1,2}^{\epsilon,\Delta t},
\]
where $\mathcal{E}_{m,n,1,1,1}^{\epsilon,\Delta t}$ is defined by
\[
\mathcal{E}_{m,n,1,1,1}^{\epsilon,\Delta t}=\frac{1}{\epsilon}\int_{t_m}^{t_{m+1}}\E[\nabla_pu^\epsilon(t_n-t,\tilde{q}^{\epsilon,\Delta t}(t_m),\tilde{p}^{\epsilon,\Delta t}(t_m))\cdot\Bigl(Df(\tilde{q}^{\epsilon,\Delta t}(t_m)).\bigl(\tilde{q}^{\epsilon,\Delta t}(t_m)-\tilde{q}^{\epsilon,\Delta t}(t)\bigr)\Bigr)]dt
\]
and where $\mathcal{E}_{m,n,1,1,2}^{\epsilon,\Delta t}$ satisfies
\begin{align*}
|\mathcal{E}_{m,n,1,1,2}^{\epsilon,\Delta t}|&\le C(T)\vvvert\varphi\vvvert_1\int_{t_m}^{t_{m+1}}\E[|\tilde{q}^{\epsilon,\Delta t}(t_m)-\tilde{q}^{\epsilon,\Delta t}(t)|^2]dt\\
&\le C(T)\vvvert\varphi\vvvert_1(1+|q_0^0|^2)\Delta t^2
\end{align*}
using first the inequality~\eqref{eq:K1} from Proposition~\ref{propo:Kolmogorov} and second the inequality~\eqref{eq:lem-regtilde} from Lemma~\ref{lem:regtilde} and Assumption~\ref{ass:init}.

To deal with the error term $\mathcal{E}_{m,n,1,1,1}^{\epsilon,\Delta t}$, a conditional expectation argument is used. For all $t\in[t_m,t_{m+1}]$, solving the auxiliary stochastic differential equation~\eqref{eq:SDEtilde}, one has
\begin{equation}\label{eq:incrementstilde}
\left\lbrace
\begin{aligned}
\tilde{q}^{\epsilon,\Delta t}(t)-\tilde{q}^{\epsilon,\Delta t}(t_m)&=\frac{1}{\epsilon}\int_{t_m}^{t}\tilde{p}^{\epsilon,\Delta t}(s)ds=\epsilon\bigl(\tilde{p}^{\epsilon,\Delta t}(t_m)-\tilde{p}^{\epsilon,\Delta t}(t)\bigr)+\int_{t_m}^{t}f(q_m^{\epsilon,\Delta t})ds+\int_{t_m}^{t}\sigma(q_m^{\epsilon,\Delta t})d\beta(s)\\
\tilde{p}^{\epsilon,\Delta t}(t)-\tilde{p}^{\epsilon,\Delta t}(t_m)&=(e^{-\frac{t-t_m}{\epsilon^2}}-1)\tilde{p}^{\epsilon,\Delta t}(t_m)+\frac{1}{\epsilon}\int_{t_m}^{t}e^{-\frac{t-s}{\epsilon^2}}f(q_m^{\epsilon,\Delta t})ds+\frac{1}{\epsilon}\int_{t_m}^{t}e^{-\frac{t-s}{\epsilon^2}}\sigma(q_m^{\epsilon,\Delta t})d\beta(s).
\end{aligned}
\right.
\end{equation}
Therefore, the error term $\mathcal{E}_{m,n,1,1,1}^{\epsilon,\Delta t}$ can be rewritten as
\begin{align*}
\mathcal{E}_{m,n,1,1,1}^{\epsilon,\Delta t}&=\int_{t_m}^{t_{m+1}}\E[\nabla_pu^\epsilon(t_n-t,\tilde{q}^{\epsilon,\Delta t}(t_m),\tilde{p}^{\epsilon,\Delta t}(t_m))\cdot\Bigl(Df(\tilde{q}^{\epsilon,\Delta t}(t_m)).\bigl(\tilde{p}^{\epsilon,\Delta t}(t)-\tilde{p}^{\epsilon,\Delta t}(t_m)\bigr)\Bigr)]dt\\
&-\frac{1}{\epsilon}\int_{t_m}^{t_{m+1}}\E[\nabla_pu^\epsilon(t_n-t,\tilde{q}^{\epsilon,\Delta t}(t_m),\tilde{p}^{\epsilon,\Delta t}(t_m))\cdot\Bigl(Df(\tilde{q}^{\epsilon,\Delta t}(t_m)).\int_{t_m}^{t}f(q_m^{\epsilon,\Delta t})ds\Bigr)]dt\\
&=\int_{t_m}^{t_{m+1}}\E[\nabla_pu^\epsilon(t_n-t,\tilde{q}^{\epsilon,\Delta t}(t_m),\tilde{p}^{\epsilon,\Delta t}(t_m))\cdot\Bigl(Df(\tilde{q}^{\epsilon,\Delta t}(t_m)).\bigl((e^{-\frac{t-t_m}{\epsilon^2}}-1)\tilde{p}^{\epsilon,\Delta t}(t_m)\bigr)\Bigr)]dt\\
&+\frac{1}{\epsilon}\int_{t_m}^{t_{m+1}}\E[\nabla_pu^\epsilon(t_n-t,\tilde{q}^{\epsilon,\Delta t}(t_m),\tilde{p}^{\epsilon,\Delta t}(t_m))\cdot\Bigl(Df(\tilde{q}^{\epsilon,\Delta t}(t_m)).\int_{t_m}^{t}\bigl(e^{-\frac{t-s}{\epsilon^2}}-1\bigr)f(q_m^{\epsilon,\Delta t})ds\Bigr)]dt,
\end{align*}
since the terms with $\int_{t_m}^{t}\sigma(q_m^{\epsilon,\Delta t})d\beta(s)$ and $\int_{t_m}^{t}e^{-\frac{t-s}{\epsilon^2}}\sigma(q_m^{\epsilon,\Delta t})d\beta(s)$ vanish in expectation.

Using the inequality~\eqref{eq:K1} from Proposition~\ref{propo:Kolmogorov}, the boundedness of $Df$, the moment bounds~\eqref{eq:moment-scheme-weak} from Proposition~\ref{propo:moment-scheme-weak} and Assumption~\ref{ass:init}, one obtains
\begin{align*}
|\mathcal{E}_{m,n,1,1,1}^{\epsilon,\Delta t}|&\le C(T)\vvvert\varphi\vvvert_1\int_{t_m}^{t_{m+1}}\epsilon\bigl(1-e^{-\frac{t-t_m}{\epsilon^2}}\bigr)dt\E[|\tilde{p}^{\epsilon,\Delta t}(t_m)|]\\
&+C(T)\vvvert\varphi\vvvert_1\Delta t^2\bigl(1+\E[|\tilde{q}^{\epsilon,\Delta t}(t_m)|]\bigr)\\
&\le C(T)\vvvert\varphi\vvvert_1\Bigl(\int_{0}^{\Delta t}\epsilon\bigl(1-e^{-\frac{t}{\epsilon^2}}\bigr)dt+\Delta t^2\Bigr)\bigl(1+|q_0^\epsilon|^2+|p_0^\epsilon|^2\bigr)\\
&\le C(T)\Delta t\vvvert\varphi\vvvert_1\Bigl(\mathcal{R}(\epsilon,\Delta t)+\Delta t\Bigr)\bigl(1+|q_0^0|^2\bigr),
\end{align*}
where we recall that $\mathcal{R}(\epsilon,\Delta t)$ is defined by~\eqref{eq:R}.

$\bullet$ Owing to the inequality~\eqref{eq:K2} from Proposition~\ref{propo:Kolmogorov}, for all $m\in\{0,\ldots,n-1\}$ and using the Lipschitz continuity of $f$, one obtains
\begin{align*}
|\mathcal{E}_{m,n,1,2}^{\epsilon,\Delta t}|&\le C(T)\vvvert\varphi\vvvert_2 \int_{t_m}^{t_{m+1}}\E[|\tilde{q}^{\epsilon,\Delta t}(t)-\tilde{q}^{\epsilon,\Delta t}(t_m)|^2]dt\\
&+C(T)\epsilon\vvvert\varphi\vvvert_2 \int_{t_m}^{t_{m+1}}\E[|\tilde{q}^{\epsilon,\Delta t}(t)-\tilde{q}^{\epsilon,\Delta t}(t_m)||\tilde{p}^{\epsilon,\Delta t}(t)-\tilde{p}^{\epsilon,\Delta t}(t_m)|]dt.
\end{align*}
Using the inequality~\eqref{eq:lem-regtilde} from Lemma~\ref{lem:regtilde} and Assumption~\ref{ass:init}, one then obtains the upper bound
\[
|\mathcal{E}_{m,n,1,2}^{\epsilon,\Delta t}|\le C(T)\vvvert\varphi\vvvert_2(1+|q_0^0|^2)\Delta t^2.
\]

$\bullet$ Since the mapping $a$ is of class $\mathcal{C}^2$ with bounded second-order derivative, a Taylor expansion argument gives
\[
\mathcal{E}_{m,n,2,1}^{\epsilon,\Delta t}=\mathcal{E}_{m,n,2,1,1}^{\epsilon,\Delta t}+\mathcal{E}_{m,n,2,1,2}^{\epsilon,\Delta t},
\]
where $\mathcal{E}_{m,n,2,1,1}^{\epsilon,\Delta t}$ is defined by
\[
\mathcal{E}_{m,n,2,1,1}^{\epsilon,\Delta t}=\frac{1}{\epsilon^2}\int_{t_m}^{t_{m+1}}\E[\nabla_{p}^2u^\epsilon(t_n-t,\tilde{q}^{\epsilon,\Delta t}(t_m),\tilde{p}^{\epsilon,\Delta t}(t_m)):\Bigl(Da(\tilde{q}^{\epsilon,\Delta t}(t_m)).\bigl(\tilde{q}^{\epsilon,\Delta t}(t_m)-\tilde{q}^{\epsilon,\Delta t}(t)\bigr)\Bigr)]dt
\]
and where $\mathcal{E}_{m,n,2,1,2}^{\epsilon,\Delta t}$ satisfies
\begin{align*}
|\mathcal{E}_{m,n,2,1,2}^{\epsilon,\Delta t}|&\le C(T)\vvvert\varphi\vvvert_2\int_{t_m}^{t_{m+1}}\E[|\tilde{q}^{\epsilon,\Delta t}(t_m)-\tilde{q}^{\epsilon,\Delta t}(t)|^2]dt\\
&\le C(T)\vvvert\varphi\vvvert_2(1+|q_0^0|^2)\Delta t^2
\end{align*}
using first the inequality~\eqref{eq:K2} from Proposition~\ref{propo:Kolmogorov} and second the inequality~\eqref{eq:lem-regtilde} from Lemma~\ref{lem:regtilde} and Assumption~\ref{ass:init}.

Like in the treatment of the error term $\mathcal{E}_{m,n,1,1,1}^{\epsilon,\Delta t}$ above, a conditional expectation argument is used to deal with the error term $\mathcal{E}_{m,n,2,1,1}^{\epsilon,\Delta t}$. The expressions~\eqref{eq:incrementstilde} of $\tilde{q}^{\epsilon,\Delta t}(t)-\tilde{q}^{\epsilon,\Delta t}(t_m)$ and $\tilde{p}^{\epsilon,\Delta t}(t)-\tilde{p}^{\epsilon,\Delta t}(t_m)$ give
\begin{align*}
\mathcal{E}_{m,n,2,1,1}^{\epsilon,\Delta t}&=\frac{1}{\epsilon}\int_{t_m}^{t_{m+1}}\E[\nabla_{p}^2u^\epsilon(t_n-t,\tilde{q}^{\epsilon,\Delta t}(t_m),\tilde{p}^{\epsilon,\Delta t}(t_m)):\Bigl(Da(\tilde{q}^{\epsilon,\Delta t}(t_m)).\bigl(\tilde{p}^{\epsilon,\Delta t}(t)-\tilde{p}^{\epsilon,\Delta t}(t_m)\bigr)\Bigr)]dt\\
&-\frac{1}{\epsilon^2}\int_{t_m}^{t_{m+1}}\E[\nabla_{p}^2u^\epsilon(t_n-t,\tilde{q}^{\epsilon,\Delta t}(t_m),\tilde{p}^{\epsilon,\Delta t}(t_m)):\Bigl(Da(\tilde{q}^{\epsilon,\Delta t}(t_m)).\int_{t_m}^{t}f(q_m^{\epsilon,\Delta t})ds\Bigr)]dt\\
&=\frac{1}{\epsilon}\int_{t_m}^{t_{m+1}}\E[\nabla_{p}^2u^\epsilon(t_n-t,\tilde{q}^{\epsilon,\Delta t}(t_m),\tilde{p}^{\epsilon,\Delta t}(t_m)):\Bigl(Da(\tilde{q}^{\epsilon,\Delta t}(t_m)).\bigl((e^{-\frac{t-t_m}{\epsilon^2}}-1)\tilde{p}^{\epsilon,\Delta t}(t_m)\bigr)\Bigr)]dt\\
&+\frac{1}{\epsilon^2}\int_{t_m}^{t_{m+1}}\E[\nabla_{p}^2u^\epsilon(t_n-t,\tilde{q}^{\epsilon,\Delta t}(t_m),\tilde{p}^{\epsilon,\Delta t}(t_m)):\Bigl(Da(\tilde{q}^{\epsilon,\Delta t}(t_m)).\int_{t_m}^{t}\bigl(e^{-\frac{t-s}{\epsilon^2}}-1\bigr)f(q_m^{\epsilon,\Delta t})ds\Bigr)]dt.
\end{align*}
Using the inequality~\eqref{eq:K2} from Proposition~\ref{propo:Kolmogorov}, the boundedness of $Da$, and the moment bounds~\eqref{eq:moment-scheme-weak} from Proposition~\ref{propo:moment-scheme-weak}, one obtains
\begin{align*}
|\mathcal{E}_{m,n,2,1,1}^{\epsilon,\Delta t}|&\le C(T)\vvvert\varphi\vvvert_2\int_{t_m}^{t_{m+1}}\epsilon\bigl(1-e^{-\frac{t-t_m}{\epsilon^2}}\bigr)dt\E[|\tilde{p}^{\epsilon,\Delta t}(t_m)|]\\
&+C(T)\vvvert\varphi\vvvert_2\Delta t^2\bigl(1+\E[|\tilde{q}^{\epsilon,\Delta t}(t_m)|]\bigr)\\
&\le C(T)\vvvert\varphi\vvvert_2\Bigl(\int_{0}^{\Delta t}\epsilon\bigl(1-e^{-\frac{t}{\epsilon^2}}\bigr)dt+\Delta t^2\Bigr)\bigl(1+|q_0^\epsilon|^2+|p_0^\epsilon|^2\bigr)\\
&\le C(T)\Delta t\vvvert\varphi\vvvert_2\Bigl(\mathcal{R}(\epsilon,\Delta t)+\Delta t\Bigr)\bigl(1+|q_0^0|^2\bigr).
\end{align*}

$\bullet$ Owing to the inequality~\eqref{eq:K3} from Proposition~\ref{propo:Kolmogorov}, for all $m\in\{0,\ldots,n-1\}$ and using the Lipschitz continuity property~\eqref{eq:aLip} of $a$, one obtains
\begin{align*}
|\mathcal{E}_{m,n,2,2}^{\epsilon,\Delta t}|&\le C(T)\vvvert\varphi\vvvert_3 \int_{t_m}^{t_{m+1}}\E\bigl[|\tilde{q}^{\epsilon,\Delta t}(t)-\tilde{q}^{\epsilon,\Delta t}(t_m)|^2\bigr]dt\\
&+C(T)\epsilon\vvvert\varphi\vvvert_3 \int_{t_m}^{t_{m+1}}\E\bigl[|\tilde{q}^{\epsilon,\Delta t}(t)-\tilde{q}^{\epsilon,\Delta t}(t_m)||\tilde{p}^{\epsilon,\Delta t}(t)-\tilde{p}^{\epsilon,\Delta t}(t_m)|\bigr]dt.
\end{align*}
Using the inequality~\eqref{eq:lem-regtilde} from Lemma~\ref{lem:regtilde} and Assumption~\ref{ass:init}, one then obtains the upper bound
\[
|\mathcal{E}_{m,n,2,2}^{\epsilon,\Delta t}|\le C(T)\vvvert\varphi\vvvert_3(1+|q_0^0|^2)\Delta t^2.
\]

$\bullet$ Gathering the estimates, one finally obtains
\begin{align*}
\big|\E[\varphi(q_n^{\epsilon,\Delta t})]-\E[\varphi(q^\epsilon(t_n))]\big|&\le \sum_{m=0}^{n-1}\bigl(|\mathcal{E}_{m,n,1}^{\epsilon,\Delta t}|+|\mathcal{E}_{m,n,2}^{\epsilon,\Delta t}|\bigr)\\
&\le C(T)\vvvert\varphi\vvvert_3\Bigl(\mathcal{R}(\epsilon,\Delta t)+\Delta t\Bigr)(1+|q_0^0|^2)
\end{align*}
for all $n\in\{0,\ldots,N-1\}$, where $\mathcal{R}(\epsilon,\Delta t)$ is defined by~\eqref{eq:R}. This concludes the proof of the weak error estimate~\eqref{eq:theo-weak} and the proof of Theorem~\ref{theo:weak}.
\end{proof}

\subsection{Proof of Proposition~\ref{propo:Kolmogorov}}\label{sec:proofKolmogorov}

\begin{proof}[Proof of the inequality~\eqref{eq:K1}]
For all ${\bf h}=(h_q,h_p)\in\R^{2d}$, one has
\[
\nabla_q u^\epsilon(t,q,p).h_q+\nabla_p u^\epsilon(t,q,p).h_p=\nabla_{(q,p)} u^\epsilon(t,q,p).{\bf h}=\E_{q,p}[\nabla_q\varphi(q^\epsilon(t)).\eta_q^{\epsilon,{\bf h}}(t)],
\]
where $t\mapsto \eta^{\epsilon,{\bf h}}(t)=\bigl(\eta_q^{\epsilon,{\bf h}}(t),\eta_p^{\epsilon,{\bf h}}(t)\bigr)\in\R^{2d}$ is solution of the stochastic differential equation
\[
\left\lbrace
\begin{aligned}
d\eta_q^{\epsilon,{\bf h}}(t)&=\frac{\eta_p^{\epsilon,{\bf h}}(t)}{\epsilon}dt\\
d\eta_p^{\epsilon,{\bf h}}(t)&=-\frac{\eta_p^{\epsilon,{\bf h}}(t)}{\epsilon^2}dt+\frac{1}{\epsilon}Df(q^\epsilon(t)).\eta_q^{\epsilon,{\bf h}}(t)dt+\frac{1}{\epsilon}D\sigma(q^\epsilon(t)).\eta_q^{\epsilon,{\bf h}}(t)d\beta(t),
\end{aligned}
\right.
\]
with initial value $\eta^{\epsilon,{\bf h}}(0)={\bf h}$, equivalently $\eta_q^{\epsilon,{\bf h}}(0)=h_q$ and $\eta_p^{\epsilon,{\bf h}}(0)=h_p$.

The inequality~\eqref{eq:K1} is a straightforward consequence of the following claim: for all $m\in\N$, there exists $C_m(T)\in(0,\infty)$ such that
\begin{equation}\label{eq:claimK1}
\underset{(t,q,p)\in[0,T]\times\R^{2d}}\sup~\E_{q,p}[|\eta_q^{\epsilon,{\bf h}}(t)|^{2m}]\le C_m(T)\bigl(|h_q|^{2m}+\epsilon^{2m}|h_p|^{2m}\bigr).
\end{equation}
Indeed, it suffices to choose $m=1$ and to set ${\bf h}=(0,h)$ to obtain~\eqref{eq:K1}. It thus remains to prove the claim~\eqref{eq:claimK1}. First, observe that, for all $t\ge 0$, one has
\[
\eta_p^{\epsilon,{\bf h}}(t)=e^{-\frac{t}{\epsilon^2}}h_p+\frac{1}{\epsilon}\int_0^t e^{-\frac{t-s}{\epsilon^2}}Df(q^\epsilon(s)).\eta_q^{\epsilon,{\bf h}}(s)ds+\frac{1}{\epsilon}\int_0^t e^{-\frac{t-s}{\epsilon^2}}D\sigma(q^\epsilon(s)).\eta_q^{\epsilon,{\bf h}}(s)d\beta(s).
\]
In addition, for all $t\ge 0$, one has
\begin{align*}
\eta_q^{\epsilon,{\bf h}}(t)&=h_q+\frac{1}{\epsilon}\int_0^{t}\eta_p^{\epsilon,{\bf h}}(s)ds\\
&=h_q+\int_0^t Df(q^\epsilon(s)).\eta_q^{\epsilon,{\bf h}}(s)ds+\int_0^t D\sigma(q^{\epsilon,{\bf h}}(s)).\eta_q^{\epsilon,{\bf h}}(s)d\beta(s)+\epsilon\bigl(\eta_p^{\epsilon,{\bf h}}(0)-\eta_p^{\epsilon,{\bf h}}(t)\bigr).
\end{align*}

On the one hand, using the H\"older and Burkholder--Davis--Gundy inequalities, one obtains the following upper bound for $\eta_p^{\epsilon,{\bf h}}(t)$: for all $t\in[0,T]$ one has
\begin{align*}
\E[|\eta_p^{\epsilon,{\bf h}}(t)|^{2m}]\le C_m(T)\Bigl(|h_p|^{2m}+\frac{1}{\epsilon^{2m}}\int_{0}^{t}\E[|\eta_q^{\epsilon,{\bf h}}(s)|^{2m}]ds\Bigr).
\end{align*}
On the other hand, using the H\"older and Burkholder--Davis--Gundy inequalities and the upper bound above for $\eta_p^{\epsilon,{\bf h}}(t)$, one obtains the following upper bound for $\eta_q^{\epsilon,{\bf h}}(t)$: for all $t\in[0,T]$ one has
\[
\E[|\eta_q^{\epsilon,{\bf h}}(t)|^{2m}]\le C_m(T)\Bigl(|h_q|^{2m}+\int_{0}^{t}\E[|\eta_q^{\epsilon,{\bf h}}(s)|^{2m}]ds+\epsilon^{2m}|h_p|^{2m}+\epsilon^{2m}\E[|\eta_p^{\epsilon,{\bf h}}(t)|^{2m}]\Bigr).
\]
Combining the two upper bounds gives the inequality
\[
\E[|\eta_q^{\epsilon,{\bf h}}(t)|^{2m}]\le C_m(T)\Bigl(|h_q|^{2m}+\epsilon^{2m}|h_p|^{2m}+\int_{0}^{t}\E[|\eta_q^{\epsilon,{\bf h}}(s)|^{2m}]ds\Bigr)
\]
for all $t\in[0,T]$. Applying Gronwall's inequality then yields the claim~\eqref{eq:claimK1}. This concludes the proof of the inequality~\eqref{eq:K1}.
\end{proof}

\begin{proof}[Proof of the inequality~\eqref{eq:K2}]
For all ${\bf h}^1=(h_q^1,h_p^1)\in\R^{2d}$ and ${\bf h}^2=(h_q^2,h_p^2)\in\R^{2d}$, one has
\begin{align*}
\nabla_{q,p}^2u^\epsilon(t,q,p).({\bf h}^1,{\bf h}^2)&=\nabla_q\nabla_qu^\epsilon(t,q,p).(h_q^1,h_q^2)+\nabla_p\nabla_pu^\epsilon(t,q,p).(h_p^1,h_p^2)\\
&+\nabla_q\nabla_pu^\epsilon(t,q,p).(h_q^1,h_p^2)+\nabla_p\nabla_qu^\epsilon(t,q,p).(h_p^1,h_q^2)\\
&=\E[\nabla_q \varphi(q^\epsilon(t)).\zeta_q^{\epsilon,{\bf h}^1,{\bf h}^2}(t)]+\E[\nabla_q^2\varphi(q^\epsilon(t)).(\eta_q^{\epsilon,{\bf h}^1}(t),\eta_q^{\epsilon,{\bf h}^2}(t))],
\end{align*}
where $t\mapsto \zeta^{\epsilon,{\bf h}^1,{\bf h}^2}(t)=\bigl(\zeta_q^{\epsilon,{\bf h}^1,{\bf h}^2}(t),\zeta_p^{\epsilon,{\bf h}^1,{\bf h}^2}(t)\bigr)\in\R^{2d}$ is solution of the stochastic differential equation
\[
\left\lbrace
\begin{aligned}
d\zeta_q^{\epsilon,{\bf h}^1,{\bf h}^2}(t)&=\frac{\zeta_p^{\epsilon,{\bf h}^1,{\bf h}^2}(t)}{\epsilon}dt\\
d\zeta_p^{\epsilon,{\bf h}^1,{\bf h}^2}(t)&=-\frac{1}{\epsilon^2}\zeta_p^{\epsilon,{\bf h}^1,{\bf h}^2}(t)dt+\frac{1}{\epsilon}Df(q^\epsilon(t)).\zeta_q^{\epsilon,{\bf h}^1,{\bf h}^2}(t)dt+\frac{1}{\epsilon}D\sigma(q^\epsilon(t)).\zeta_q^{\epsilon,{\bf h}^1,{\bf h}^2}d\beta(t)\\
&+\frac{1}{\epsilon}D^2f(q^\epsilon(t)).(\eta_q^{\epsilon,{\bf h}^1}(t),\eta_q^{\epsilon,{\bf h}^2}(t))dt+\frac{1}{\epsilon}D^2\sigma(q^{\epsilon}(t)).(\eta_q^{\epsilon,{\bf h}^1}(t),\eta_q^{\epsilon,{\bf h}^2}(t))d\beta(t),
\end{aligned}
\right.
\]
with initial value $\zeta^{\epsilon,{\bf h}^1,{\bf h}^2}(0)=0$, equivalently $\zeta_q^{\epsilon,{\bf h}^1,{\bf h}^2}(0)=\zeta_p^{\epsilon,{\bf h}^1,{\bf h}^2}(0)=0$.

The inequality~\eqref{eq:K2} is a straightforward consequence of the inequality~\eqref{eq:K1} and of the following claim: for all $m\in\N$, there exists $C_m(T)\in(0,\infty)$ such that
\begin{equation}\label{eq:claimK2}
\underset{(t,q,p)\in[0,T]\times\R^{2d}}\sup~\E_{q,p}[|\zeta_q^{\epsilon,{\bf h}^1,{\bf h}^2}(t)|^{2m}]\le C_m(T)\bigl(|h_q^1|^{2m}+\epsilon^{2m}|h_p^1|^{2m}\bigr)\bigl(|h_q^2|^{2m}+\epsilon^{2m}|h_p^2|^{2m}\bigr).
\end{equation}
Indeed, it suffices to apply the inequality~\eqref{eq:claimK1} with $m=1$ combined with the Cauchy--Schwarz inequality, and the inequality~\eqref{eq:claimK2} with $m=1$, and to set either ${\bf h}^1=(h^1,0), {\bf h}^2=(0,h^2)$, or ${\bf h}^1=(0,h^1), {\bf h}^2=(0,h^2)$, to obtain~\eqref{eq:K2}. It thus remains to prove the claim~\eqref{eq:claimK2}. First, observe that, for all $t\ge 0$, one has
\begin{align*}
\zeta_p^{\epsilon,{\bf h}^1,{\bf h}^2}(t)&=\frac{1}{\epsilon}\int_{0}^{t}e^{-\frac{t-s}{\epsilon^2}}Df(q^\epsilon(s)).\zeta_q^{\epsilon,{\bf h}^1,{\bf h}^2}(s)ds+\frac{1}{\epsilon}\int_{0}^{t}e^{-\frac{t-s}{\epsilon^2}}D\sigma(q^\epsilon(s)).\zeta_q^{\epsilon,{\bf h}^1,{\bf h}^2}(s)d\beta(s)\\
&+\frac{1}{\epsilon}\int_{0}^{t}e^{-\frac{t-s}{\epsilon^2}}D^2f(q^\epsilon(s)).(\eta_q^{\epsilon,{\bf h}^1}(s),\eta_q^{\epsilon,{\bf h}^2}(s))ds+\frac{1}{\epsilon}\int_{0}^{t}e^{-\frac{t-s}{\epsilon^2}}D^2\sigma(q^\epsilon(s)).(\eta_q^{\epsilon,{\bf h}^1}(s),\eta_q^{\epsilon,{\bf h}^2}(s))d\beta(s).
\end{align*}
In addition, for all $t\ge 0$, one has
\begin{align*}
\zeta_q^{\epsilon,{\bf h}^1,{\bf h}^2}(t)=\frac{1}{\epsilon}\int_{0}^{t}&\zeta_p^{\epsilon,{\bf h}^1,{\bf h}^2}(s)ds=-\epsilon \zeta_p^{\epsilon,{\bf h}^1,{\bf h}^2}(t)\\
&+\int_{0}^{t}Df(q^\epsilon(s)).\zeta_q^{\epsilon,{\bf h}^1,{\bf h}^2}(s)ds+\int_{0}^{t}D\sigma(q^\epsilon(s)).\zeta_q^{\epsilon,{\bf h}^1,{\bf h}^2}(s)d\beta(s)\\
&+\int_{0}^{t}D^2f(q^\epsilon(s)).(\eta_q^{\epsilon,{\bf h}^1}(s),\eta_q^{\epsilon,{\bf h}^2}(s))ds+\int_{0}^{t}D^2\sigma(q^\epsilon(s)).(\eta_q^{\epsilon,{\bf h}^1}(s),\eta_q^{\epsilon,{\bf h}^2}(s))d\beta(s).
\end{align*}
On the one hand, using the H\"older and Burkholder--Davis--Gundy inequalities, one obtains the following upper bound for $\zeta_p^{\epsilon,{\bf h}^1,{\bf h}^2}(t)$: for all $t\in[0,T]$ one has
\begin{align*}
\E[|\zeta_p^{\epsilon,{\bf h}^1,{\bf h}^2}(t)|^{2m}]&\le \frac{C_m(T)}{\epsilon^{2m}}\Bigl(\int_{0}^{t}\E[|\zeta_q^{\epsilon,{\bf h}^1,{\bf h}^2}(s)|^{2m}]ds+\int_0^t\E[|\eta_q^{\epsilon,{\bf h}^1}(s)|^{2m}|\eta_q^{\epsilon,{\bf h}^2}(s)|^{2m}]ds\Bigr)\\
&\le \frac{C_m(T)}{\epsilon^{2m}}\Bigl(\int_{0}^{t}\E[|\zeta_q^{\epsilon,{\bf h}^1,{\bf h}^2}(s)|^{2m}]ds+\bigl(|h_q^1|^{2m}+\epsilon^{2m}|h_p^1|^{2m}\bigr)\bigl(|h_q^2|^{2m}+\epsilon^{2m}|h_p^2|^{2m}\bigr)\Bigr),
\end{align*}
where the inequality~\eqref{eq:claimK1} is used in the second step above.

On the other hand, using the H\"older and Burkholder--Davis--Gundy inequalities, one obtains the following upper bound for $\zeta_q^{\epsilon,{\bf h}^1,{\bf h}^2}(t)$: for all $t\in[0,T]$ one has
\begin{align*}
\E[|\zeta_q^{\epsilon,{\bf h}^1,{\bf h}^2}(t)|^{2m}]
&\le C_m(T)\epsilon^{2m}\E[|\zeta_p^{\epsilon,{\bf h}^1,{\bf h}^2}(t)|^{2m}]\\
&+C_m(T)\int_{0}^{t}\E[|\zeta_q^{\epsilon,{\bf h}^1,{\bf h}^2}(s)|^{2m}]ds+C_m(T)\int_0^t\E[|\eta_q^{\epsilon,{\bf h}^1}(s)|^{2m}|\eta_q^{\epsilon,{\bf h}^2}(s)|^{2m}]ds\\
&\le C_m(T)\epsilon^{2m}\E[|\zeta_p^{\epsilon,{\bf h}^1,{\bf h}^2}(t)|^{2m}]\\
&+C_m(T)\Bigl(\int_{0}^{t}\E[|\zeta_q^{\epsilon,{\bf h}^1,{\bf h}^2}(s)|^{2m}]ds+\bigl(|h_q^1|^{2m}+\epsilon^{2m}|h_p^1|^{2m}\bigr)\bigl(|h_q^2|^{2m}+\epsilon^{2m}|h_p^2|^{2m}\bigr)\Bigr),
\end{align*}
where the inequality~\eqref{eq:claimK1} is used in the second step above.

Combining the two upper bounds gives the inequality
\[
\E[|\zeta_q^{\epsilon,{\bf h}^1,{\bf h}^2}(t)|^{2m}]\le C_m(T)\Bigl(\int_{0}^{t}\E[|\zeta_q^{\epsilon,{\bf h}^1,{\bf h}^2}(s)|^{2m}]ds+\bigl(|h_q^1|^{2m}+\epsilon^{2m}|h_p^1|^{2m}\bigr)\bigl(|h_q^2|^{2m}+\epsilon^{2m}|h_p^2|^{2m}\bigr)\Bigr)
\]
for all $t\in[0,T]$. Applying Gronwall's inequality then yields the claim~\eqref{eq:claimK2}. This concludes the proof of the inequality~\eqref{eq:K2}.
\end{proof}

\begin{proof}[Proof of the inequality~\eqref{eq:K3}]
For all ${\bf h}^1=(h_q^1,h_p^1)\in\R^{2d}$, ${\bf h}^2=(h_q^2,h_p^2)\in\R^{2d}$ and ${\bf h}^3=(h_q^3,h_p^3)\in\R^{2d}$, one has
\begin{align*}
\nabla_{q,p}^3u^\epsilon(t,q,p).({\bf h}^1,{\bf h}^2,{\bf h}^3)&=\E[\nabla_q \varphi(q^\epsilon(t)).\xi_q^{\epsilon,{\bf h}^1,{\bf h}^2,{\bf h}^3}(t)]\\
&+\E[\nabla_q^2\varphi(q^\epsilon(t)).(\eta_q^{\epsilon,{\bf h}^1}(t),\zeta_q^{\epsilon,{\bf h}^2,{\bf h}^3}(t))]\\
&+\E[\nabla_q^2\varphi(q^\epsilon(t)).(\eta_q^{\epsilon,{\bf h}^2}(t),\zeta_q^{\epsilon,{\bf h}^3,{\bf h}^1}(t))]\\
&+\E[\nabla_q^2\varphi(q^\epsilon(t)).(\eta_q^{\epsilon,{\bf h}^3}(t),\zeta_q^{\epsilon,{\bf h}^1,{\bf h}^2}(t))]\\
&+\E[\nabla_q^3\varphi(q^\epsilon(t)).\bigl(\eta_q^{\epsilon,{\bf h}^1}(t),\eta_q^{\epsilon,{\bf h}^2}(t),\eta_q^{\epsilon,{\bf h}^3}(t)\bigr)]
\end{align*}
where $t\mapsto \xi^{\epsilon,{\bf h}^1,{\bf h}^2,{\bf h}^3}(t)=\bigl(\xi_q^{\epsilon,{\bf h}^1,{\bf h}^2,{\bf h}^3}(t),\xi_p^{\epsilon,{\bf h}^1,{\bf h}^2,{\bf h}^3}(t)\bigr)\in\R^{2d}$ is solution of the stochastic differential equation
\[
\left\lbrace
\begin{aligned}
d\xi_q^{\epsilon,{\bf h}^1,{\bf h}^2,{\bf h}^3}(t)&=\frac{\xi_p^{\epsilon,{\bf h}^1,{\bf h}^2,{\bf h}^3}(t)}{\epsilon}dt\\
d\xi_p^{\epsilon,{\bf h}^1,{\bf h}^2,{\bf h}^3}(t)&=-\frac{1}{\epsilon^2}\xi_p^{\epsilon,{\bf h}^1,{\bf h}^2,{\bf h}^3}(t)dt+\frac{1}{\epsilon}Df(q^\epsilon(t)).\xi_q^{\epsilon,{\bf h}^1,{\bf h}^2,{\bf h}^3}(t)dt+\frac{1}{\epsilon}D\sigma(q^\epsilon(t)).\xi_q^{\epsilon,{\bf h}^1,{\bf h}^2,{\bf h}^3}d\beta(t)\\
&+\frac{1}{\epsilon}D^2f(q^\epsilon(t)).(\eta_q^{\epsilon,{\bf h}^1}(t),\zeta_q^{\epsilon,{\bf h}^2,{\bf h}^3}(t))dt+\frac{1}{\epsilon}D^2\sigma(q^{\epsilon}(t)).(\eta_q^{\epsilon,{\bf h}^1}(t),\zeta_q^{\epsilon,{\bf h}^2,{\bf h}^3}(t))d\beta(t)\\
&+\frac{1}{\epsilon}D^2f(q^\epsilon(t)).(\eta_q^{\epsilon,{\bf h}^2}(t),\zeta_q^{\epsilon,{\bf h}^3,{\bf h}^1}(t))dt+\frac{1}{\epsilon}D^2\sigma(q^{\epsilon}(t)).(\eta_q^{\epsilon,{\bf h}^2}(t),\zeta_q^{\epsilon,{\bf h}^2,{\bf h}^3}(t))d\beta(t)\\
&+\frac{1}{\epsilon}D^2f(q^\epsilon(t)).(\eta_q^{\epsilon,{\bf h}^3}(t),\eta_q^{\epsilon,{\bf h}^1,{\bf h}^2}(t))dt+\frac{1}{\epsilon}D^2\sigma(q^{\epsilon}(t)).(\eta_q^{\epsilon,{\bf h}^3}(t),\zeta_q^{\epsilon,{\bf h}^1,{\bf h}^2}(t))d\beta(t)\\
&+\frac{1}{\epsilon}D^3f(q^\epsilon(t)).\bigl(\eta_q^{\epsilon,{\bf h}^1}(t),\eta_q^{\epsilon,{\bf h}^2}(t),\eta_q^{\epsilon,{\bf h}^3}(t)\bigr)dt+\frac{1}{\epsilon}D^3\sigma(q^{\epsilon}(t)).\bigl(\eta_q^{\epsilon,{\bf h}^1}(t),\eta_q^{\epsilon,{\bf h}^2}(t),\eta_q^{\epsilon,{\bf h}^3}(t)\bigr)d\beta(t)
\end{aligned}
\right.
\]
with initial value $\xi^{\epsilon,{\bf h}^1,{\bf h}^2,{\bf h}^3}(0)=0$, equivalently $\zeta_q^{\epsilon,{\bf h}^1,{\bf h}^2,{\bf h}^3}(0)=\zeta_p^{\epsilon,{\bf h}^1,{\bf h}^2,{\bf h}^3}(0)=0$.

The inequality~\eqref{eq:K3} is a straightforward consequence of the inequality~\eqref{eq:K1} and of the following claim: for all $m\in\N$, there exists $C_m(T)\in(0,\infty)$ such that
\begin{equation}\label{eq:claimK3}
\underset{(t,q,p)\in[0,T]\times\R^{2d}}\sup~\E_{q,p}[|\xi_q^{\epsilon,{\bf h}^1,{\bf h}^2,{\bf h}^3}(t)|^{2m}]\le C_m(T)\bigl(|h_q^1|^{2m}+\epsilon^{2m}|h_p^1|^{2m}\bigr)\bigl(|h_q^2|^{2m}+\epsilon^{2m}|h_p^2|^{2m}\bigr)\bigl(|h_q^3|^{2m}+\epsilon^{2m}|h_p^3|^{2m}\bigr).
\end{equation}
Indeed, it suffices to apply the inequality~\eqref{eq:claimK1} with $m=2$ combined with the Cauchy--Schwarz inequality, the inequality~\eqref{eq:claimK2} with $m=1$ and the inequality~\eqref{eq:claimK3} with $m=1$, and to set either ${\bf h}^1=(h^1,0)$, ${\bf h}^2=(0,h^2)$, ${\bf h}^3=(0,h^3)$ or ${\bf h}^1=(0,h^1)$, ${\bf h}^2=(0,h^2)$, ${\bf h}^3=(0,h^3)$. It thus remains to prove the claim~\eqref{eq:claimK2}.

First, observe that, for all $t\ge 0$, one has
\begin{align*}
&\xi_p^{\epsilon,{\bf h}^1,{\bf h}^2,{\bf h}^3}(t)=\frac{1}{\epsilon}\int_{0}^{t}e^{-\frac{t-s}{\epsilon^2}}Df(q^\epsilon(s)).\xi_q^{\epsilon,{\bf h}^1,{\bf h}^2,{\bf h}^3}(s)ds+\frac{1}{\epsilon}\int_{0}^{t}e^{-\frac{t-s}{\epsilon^2}}D\sigma(q^\epsilon(s)).\xi_q^{\epsilon,{\bf h}^1,{\bf h}^2,{\bf h}^3}(s)d\beta(s)\\
&+\frac{1}{\epsilon}\int_{0}^{t}e^{-\frac{t-s}{\epsilon^2}}D^2f(q^\epsilon(s)).(\eta_q^{\epsilon,{\bf h}^1}(s),\zeta_q^{\epsilon,{\bf h}^2,{\bf h}^3}(s))ds+\frac{1}{\epsilon}\int_{0}^{t}e^{-\frac{t-s}{\epsilon^2}}D^2\sigma(q^\epsilon(s)).(\eta_q^{\epsilon,{\bf h}^1}(s),\zeta_q^{\epsilon,{\bf h}^2,{\bf h}^3}(s))d\beta(s)\\
&+\frac{1}{\epsilon}\int_{0}^{t}e^{-\frac{t-s}{\epsilon^2}}D^2f(q^\epsilon(s)).(\eta_q^{\epsilon,{\bf h}^2}(s),\eta_q^{\epsilon,{\bf h}^3,{\bf h}^1}(s))ds+\frac{1}{\epsilon}\int_{0}^{t}e^{-\frac{t-s}{\epsilon^2}}D^2\sigma(q^\epsilon(s)).(\eta_q^{\epsilon,{\bf h}^2}(s),\zeta_q^{\epsilon,{\bf h}^3,{\bf h}^1}(s))d\beta(s)\\
&+\frac{1}{\epsilon}\int_{0}^{t}e^{-\frac{t-s}{\epsilon^2}}D^2f(q^\epsilon(s)).(\eta_q^{\epsilon,{\bf h}^3}(s),\eta_q^{\epsilon,{\bf h}^1,{bf h}^2}(s))ds+\frac{1}{\epsilon}\int_{0}^{t}e^{-\frac{t-s}{\epsilon^2}}D^2\sigma(q^\epsilon(s)).(\eta_q^{\epsilon,{\bf h}^3}(s),\eta_q^{\epsilon,{\bf h}^1,{\bf h}^2}(s))d\beta(s)\\
&+\frac{1}{\epsilon}\int_{0}^{t}e^{-\frac{t-s}{\epsilon^2}}D^3f(q^\epsilon(s)).(\eta_q^{\epsilon,{\bf h}^1}(s),\eta_q^{\epsilon,{\bf h}^2}(s),\eta_q^{\epsilon,{\bf h}^3}(s))ds\\
&+\frac{1}{\epsilon}\int_{0}^{t}e^{-\frac{t-s}{\epsilon^2}}D^3\sigma(q^\epsilon(s)).(\eta_q^{\epsilon,{\bf h}^1}(s),\eta_q^{\epsilon,{\bf h}^2}(s),\eta_q^{\epsilon,{\bf h}^3}(s))d\beta(s).
\end{align*}
In addition, for all $t\ge 0$, one has
\begin{align*}
\xi_q^{\epsilon,{\bf h}^1,{\bf h}^2,{\bf h}^3}(t)=\frac{1}{\epsilon}\int_{0}^{t}&\xi_p^{\epsilon,{\bf h}^1,{\bf h}^2,{\bf h}^3}(s)ds=-\epsilon \xi_p^{\epsilon,{\bf h}^1,{\bf h}^2,{\bf h}^3}(t)\\
&+\int_{0}^{t}Df(q^\epsilon(s)).\xi_q^{\epsilon,{\bf h}^1,{\bf h}^2,{\bf h}^3}(s)ds+\int_{0}^{t}D\sigma(q^\epsilon(s)).\xi_q^{\epsilon,{\bf h}^1,{\bf h}^2,{\bf h}^3}(s)d\beta(s)\\
&+\int_{0}^{t}D^2f(q^\epsilon(s)).(\eta_q^{\epsilon,{\bf h}^1}(s),\zeta_q^{\epsilon,{\bf h}^2,{\bf h}^3}(s))ds+\int_{0}^{t}D^2\sigma(q^\epsilon(s)).(\eta_q^{\epsilon,{\bf h}^1}(s),\zeta_q^{\epsilon,{\bf h}^2,{\bf h}^3}(s))d\beta(s)\\
&+\int_{0}^{t}D^2f(q^\epsilon(s)).(\eta_q^{\epsilon,{\bf h}^2}(s),\eta_q^{\epsilon,{\bf h}^3,{\bf h}^1}(s))ds+\int_{0}^{t}D^2\sigma(q^\epsilon(s)).(\eta_q^{\epsilon,{\bf h}^2}(s),\zeta_q^{\epsilon,{\bf h}^3,{\bf h}^1}(s))d\beta(s)\\
&+\int_{0}^{t}D^2f(q^\epsilon(s)).(\eta_q^{\epsilon,{\bf h}^3}(s),\eta_q^{\epsilon,{\bf h}^1,{bf h}^2}(s))ds+\int_{0}^{t}D^2\sigma(q^\epsilon(s)).(\eta_q^{\epsilon,{\bf h}^3}(s),\eta_q^{\epsilon,{\bf h}^1,{\bf h}^2}(s))d\beta(s)\\
&+\int_{0}^{t}D^3f(q^\epsilon(s)).(\eta_q^{\epsilon,{\bf h}^1}(s),\eta_q^{\epsilon,{\bf h}^2}(s),\eta_q^{\epsilon,{\bf h}^3}(s))ds\\
&+\int_{0}^{t}D^3\sigma(q^\epsilon(s)).(\eta_q^{\epsilon,{\bf h}^1}(s),\eta_q^{\epsilon,{\bf h}^2}(s),\eta_q^{\epsilon,{\bf h}^3}(s))d\beta(s).
\end{align*}
On the one hand, using the H\"older and Burkholder--Davis--Gundy inequalities, one obtains the following upper bound for $\xi_p^{\epsilon,{\bf h}^1,{\bf h}^2,{\bf h}^3}(t)$: for all $t\in[0,T]$ one has
\begin{align*}
\E[&|\xi_p^{\epsilon,{\bf h}^1,{\bf h}^2,{\bf h}^3}(t)|^{2m}]\le \frac{C_m(T)}{\epsilon^{2m}}\int_{0}^{t}\E[|\xi_q^{\epsilon,{\bf h}^1,{\bf h}^2,{\bf h}^3}(s)|^{2m}]ds\\
&+\frac{C_m(T)}{\epsilon^{2m}}\int_0^t\bigl(\E[|\eta_q^{\epsilon,{\bf h}^1}(s)|^{2m}|\zeta_q^{\epsilon,{\bf h}^2,{\bf h}^3}(s)|^{2m}]+\E[|\eta_q^{\epsilon,{\bf h}^2}(s)|^{2m}|\zeta_q^{\epsilon,{\bf h}^3,{\bf h}^1}(s)|^{2m}]+\E[|\eta_q^{\epsilon,{\bf h}^3}(s)|^{2m}|\zeta_q^{\epsilon,{\bf h}^1,{\bf h}^2}(s)|^{2m}]\bigr)ds\\
&+\frac{C_m(T)}{\epsilon^{2m}}\int_0^t\E[|\eta_q^{\epsilon,{\bf h}^1}(s)|^{2m}|\eta_q^{\epsilon,{\bf h}^2}(s)|^{2m}|\eta_q^{\epsilon,{\bf h}^3}(s)|^{2m}]ds\\
&\le \frac{C_m(T)}{\epsilon^{2m}}\Bigl(\int_{0}^{t}\E[|\xi_q^{\epsilon,{\bf h}^1,{\bf h}^2,{\bf h}^3}(s)|^{2m}]ds+\bigl(|h_q^1|^{2m}+\epsilon^{2m}|h_p^1|^{2m}\bigr)\bigl(|h_q^2|^{2m}+\epsilon^{2m}|h_p^2|^{2m}\bigr)\bigl(|h_q^3|^{2m}+\epsilon^{2m}|h_p^3|^{2m}\bigr)\Bigr),
\end{align*}
where the inequalities~\eqref{eq:claimK1} and~\eqref{eq:claimK2} are used in the second step above.

On the other hand, using the H\"older and Burkholder--Davis--Gundy inequalities, one obtains the following upper bound for $\xi_q^{\epsilon,{\bf h}^1,{\bf h}^2,{\bf h}^3}(t)$: for all $t\in[0,T]$ one has
\begin{align*}
\E[|\xi_q^{\epsilon,{\bf h}^1,{\bf h}^2,{\bf h}^3}(t)|^{2m}]
&\le C_m(T)\epsilon^{2m}\E[|\xi_p^{\epsilon,{\bf h}^1,{\bf h}^2,{\bf h}^3}(t)|^{2m}]\\
&+C_m(T)\int_{0}^{t}\E[|\xi_q^{\epsilon,{\bf h}^1,{\bf h}^2,{\bf h}^3}(s)|^{2m}]ds\\
&+C_m(T)\int_0^t\E[|\eta_q^{\epsilon,{\bf h}^1}(s)|^{2m}|\zeta_q^{\epsilon,{\bf h}^2,{\bf h}^3}(s)|^{2m}]ds\\
&+C_m(T)\int_0^t\E[|\eta_q^{\epsilon,{\bf h}^2}(s)|^{2m}|\zeta_q^{\epsilon,{\bf h}^3,{\bf h}^1}(s)|^{2m}]ds\\
&+C_m(T)\int_0^t\E[|\eta_q^{\epsilon,{\bf h}^3}(s)|^{2m}|\zeta_q^{\epsilon,{\bf h}^1,{\bf h}^2}(s)|^{2m}]ds\\
&+C_m(T)\int_0^t\E[|\eta_q^{\epsilon,{\bf h}^1}(s)|^{2m}|\eta_q^{\epsilon,{\bf h}^2}(s)|^{2m}|\eta_q^{\epsilon,{\bf h}^3}(s)|^{2m}]ds\\
&\le C_m(T)\epsilon^{2m}\E[|\xi_p^{\epsilon,{\bf h}^1,{\bf h}^2,{\bf h}^3}(t)|^{2m}]\\
&+C_m(T)\int_{0}^{t}\E[|\xi_q^{\epsilon,{\bf h}^1,{\bf h}^2,{\bf h}^3}(s)|^{2m}]ds\\
&+\bigl(|h_q^1|^{2m}+\epsilon^{2m}|h_p^1|^{2m}\bigr)\bigl(|h_q^2|^{2m}+\epsilon^{2m}|h_p^2|^{2m}\bigr)\bigl(|h_q^3|^{2m}+\epsilon^{2m}|h_p^3|^{2m}\bigr)\Bigr),
\end{align*}
where the inequalities~\eqref{eq:claimK1} and~\eqref{eq:claimK2} are used in the second step above.

Combining the two upper bounds gives the inequality
\begin{align*}
\E[|\xi_q^{\epsilon,{\bf h}^1,{\bf h}^2,{\bf h}^3}(t)|^{2m}]&\le C_m(T)\int_{0}^{t}\E[|\xi_q^{\epsilon,{\bf h}^1,{\bf h}^2,{\bf h}^3}(s)|^{2m}]ds\\
&+C_m(T)\Bigl(\bigl(|h_q^1|^{2m}+\epsilon^{2m}|h_p^1|^{2m}\bigr)\bigl(|h_q^2|^{2m}+\epsilon^{2m}|h_p^2|^{2m}\bigr)\bigl(|h_q^3|^{2m}+\epsilon^{2m}|h_p^3|^{2m}\bigr)\Bigr)
\end{align*}
for all $t\in[0,T]$. Applying Gronwall's inequality then yields the claim~\eqref{eq:claimK3}. This concludes the proof of the inequality~\eqref{eq:K3}.

\end{proof}

\section*{Acknowledgements}
This work is partially supported by the projects ADA (ANR-19-CE40-0019-02) and SIMALIN (ANR-19-CE40-0016) operated by the French National Research Agency.


\end{document}